\documentclass{proc-l}
\usepackage{graphicx,color}
\usepackage{amsmath,amssymb,amsthm}
\usepackage{subfigure}
\usepackage{float}
\usepackage{lineno,hyperref}
\usepackage{multicol}

\modulolinenumbers[5]

\pdfoutput=1

\newtheorem{thm}{Theorem}[section]
\newtheorem{lem}[thm]{Lemma}
\newtheorem{prop}[thm]{Proposition}
\newtheorem{cor}[thm]{Corollary}

\newtheorem{rem}[thm]{Remark}

\begin{document}
%%%%%%%%%%%%%%%%%%%%%%%%%%%%%%%%%%%%%%

\title[Local minimizers over the Nehari manifold]{Local minimizers over the Nehari manifold for a class of concave-convex problems with sign changing nonlinearity}

\author[Kaye Silva]{Kaye Silva } 
\email{kayeoliveira@hotmail.com}

\author[Abiel Macedo]{Abiel Macedo}
\email{abielcosta@gmail.com}

\address[UFG]{Instituto de Matem\'{a}tica e Estat\'istica. Universidade Federal de Goi\'as, 74001-970 Goi\^ania, GO, Brazil}

	

\maketitle
%%%%%%%%%%%%%%%%%%%%%%%%%%%%%%%%%%%%%%
\begin{abstract} We study a $p$-Laplacian equation involving a parameter $\lambda$ and a concave-convex nonlinearity containing a weight which can change sign. By using the Nehari manifold and the fibering method, we show the existence of two positive solutions on some interval $(0,\lambda^*+\varepsilon)$, where $\lambda^*$ can be characterized variationally. We also study the asymptotic behavior of solutions when $\lambda\downarrow 0$.

\end{abstract}

%%%%%%%%%%%%%%%%%%%%%%%%%%%%%%%%%%%%%%%%%%%%%%%%%%%%%%%%%%%%%%%%%%%%%%%%%%%%%%%%%%%%%%%%%%%%%%%%%%%%%%%%%%%%%%55
\section{Introduction}

Consider the following equation

\begin{equation}\label{pq}\tag{$p,q,\gamma$}
\left\{
\begin{aligned}
-\Delta_p u &= \lambda |u|^{q-2}u+f|u|^{\gamma-2}u
&\mbox{in}\ \ \Omega, \\
&u\in W_0^{1,p}(\Omega), \nonumber
\end{aligned}
\right.
\end{equation}
where $\Omega\subset \mathbb{R}^N$ is a bounded domain with $C^1$ boundary, $\lambda>0$, $1<q<p<\gamma<p^*$ and $p^*$  is the critical Sobolev exponent, $f\in L^\infty(\Omega)$ and $W_0^{1,p}(\Omega)$ is the standard Sobolev space. We say that $u\in  W_0^{1,p}(\Omega)$ is a solution of (\ref{pq}) if $u$ is a critical point for $\Phi_{\lambda}:  W_0^{1,p}(\Omega)\to \mathbb{R}$ where

$$\label{philambda}
\Phi_{\lambda}(u):=\frac{1}{p}\int |\nabla u|^p-\frac{\lambda}{q}\int |u|^q-\frac{1}{\gamma}\int f|u|^\gamma.
$$

We denote $\|u\|=\left(\int |\nabla u|^p\right)^{1/p}$ as the standard Sobolev norm in $W_0^{1,p}(\Omega)$ and consider the following extremal value

$$
\lambda^*\equiv \frac{\gamma-p}{\gamma-q}\left(\frac{p-q}{\gamma-q}\right)^{\frac{p-q}{\gamma-p}} \inf_{u\in W_0^{1,p}\setminus\{0\}} \left\{\frac{\|u\|^{p\frac{\gamma-q}{\gamma-p}}}{\|u\|_q^qF(u)^{\frac{p-q}{\gamma-p}}}:\ F(u)>0 \right\},
$$
where $F(u)=\int f|u|^\gamma$. Let $z\in W_0^{1,p}(\Omega)$ be the unique positive solution of the Lane-Emden equation

\begin{equation*}
\left\{
\begin{aligned}
-\Delta_p u &=  |u|^{q-2}u
&\mbox{in}\ \ \Omega, \\
&u\in W_0^{1,p}(\Omega). \nonumber
\end{aligned}
\right.
\end{equation*}

The main result of this work is the following

\begin{thm}\label{THM} Assume that $f^+:= \max\{f(x),0\}\not\equiv 0$. There exists $\varepsilon>0$ such that for all $\lambda\in (0,\lambda^*+\varepsilon)$ the problem \eqref{pq} has two positive solutions $w_\lambda,u_\lambda$. Moreover 
	\begin{description}
		\item[(i)] $D_{uu}\Phi_{\lambda}(w_\lambda)(w_\lambda,w_\lambda)<0$, $D_{uu}\Phi_{\lambda}(u_\lambda)(u_\lambda,u_\lambda)>0$;
		\item[(ii)] 
		$$
		\lim_{\lambda\downarrow 0}\frac{u_\lambda}{\lambda^{\frac{1}{p-q}}}=z.
		$$ 
	\end{description}
\end{thm}

When $p=2$ and $f\equiv 1$ the problem \eqref{pq} was studied by Ambrosetti-Brezis-Cerami in \cite{abc}. There, among other things, they show the existence of $\Lambda>0$ such that for all $\lambda\in(0,\Lambda)$ the problem \eqref{pq} has at least two positive solutions while if $\lambda=\Lambda$ it has at least one positive solution and for $\lambda>\Lambda$ there is no positive solution for \eqref{pq}.  To find the first solution they used sub and super solution method while for the second solution they used the mountain pass theorem. Moreover, from the sub and super solution method, one can easily see that the first branch of solutions wich bifurcates from $0$ satisfies property $\mathbf{(ii)}$. Later on, there was some improvement in Ambrosetti-Azorero-Peral \cite{amazopera}, where the authors proved the existence of some $\Lambda$ satisfying the above properties, however, for $p>1$, $f\equiv 1$ and $\Omega$ a ball. Finally, the result was generalized for $p>1$ by Azorero-Peral-Manfredi in \cite{azoperaman}. 

More recently, some authors studied the problem \eqref{pq} by using only variational methods, to wit, the Nehari manifold (see Nehari \cite{neh,neh1}) and the fibering method of Pohozaev \cite{poh}. Among these authors we can cite the work of Il'yasov \cite{ilyas1}, which considered the problem \eqref{pq} with $0\le f\in L^d(\Omega)$ and $p>1$. He was able to show the existence of a parameter $\lambda^*>0$ such that for each $\lambda\in (0,\lambda^*)$ the problem \eqref{pq} has two positive solutions. In \cite{brownwu} Brown-Wu considered the case $p=2$ and a indefinite nonlinearity, that is, $f$ change sign in $\Omega$. By minimizing over the Nehari manifold they proved the existence of two positive solutions for small $\lambda$.

On the same direction, in \cite{ilyasENMM} Il'yasov provided a general theory by considering a generalization of the Rayleigh quotient, where one is able to show the existence of solutions to nonlinear elliptic equations depending on a parameter $\lambda$. In the theory, the above mentioned parameter $\lambda^*$ is called an extremal value and if $\mathcal{N}_\lambda$ is the Nehari manifold associated with \eqref{pq} then for all $\lambda\in (0,\lambda^*)$ we have that $\mathcal{N}_\lambda$ is a $C^1$ manifold with codimension $1$. These extremal values are not new and can be found for example in Ouyang \cite{ou2}. When $\lambda\in (0,\lambda^*)$, by using standard minimization techniques, one can easily minimize the energy functional associated with \eqref{pq} over the Nehari manifold, however, when $\lambda\ge \lambda^*$ things get complicated because $\mathcal{N}_\lambda$ is no longer a manifold and a finer investigation has to be done.  

Our objective in this work is to study problem \eqref{pq} only by variational methods, in particular, we use the Nehari manifold and the fibering approach. We analyze the case where $f^+\equiv \max\{f(x),0\}\not\equiv 0$ and give a contribution on the understanding of the extreme Nehari manifold $\mathcal{N}_{\lambda^*}$. Minizing over a submanifold of the Nehari manifold $\mathcal{N}_\lambda$ we show the existence of solutions for $\lambda$ near $\lambda^*$.

In Section 2 we collect some technical results. In Section 3 we show existence of two positive solutions for $\lambda\in [0,\lambda^*]$. In Section 4 we show existence of two positive solutions for $\lambda\in (\lambda^*,\lambda^*+\varepsilon)$. In Section 5 we study the asymptotic behavior for one of the branches of solutions as $\lambda\downarrow  0$. In Section 6 we prove Theorem \ref{THM}. In the Appendix, we prove some auxiliary results and we present a table with the main notations which are used throughout the work.

In this paper, $c,C$ denotes positive constants which can change from line to line, however, they depend only on $p,q,\gamma$, $\Omega$, $f$ and its dependence on these parameters are not important for the development of the work.

%%%%%%%%%%%%%%%%%%%%%%%%%%%%%%%%%%%%%%%%%%%%%%%%%%%%%%%%%%%%%%%%%%%%%%%%%%%%%%%%%%%%%%%%%%%%%%%%%%%%%%%%%
\section{Technical Results}

\label{pagen}
In this section, we collect some technical results. Consider the Nehari manifold associated to the functional $\Phi_{\lambda}$ (see Nehari \cite{neh,neh1})

$$
\mathcal{N}_{\lambda}=\left\{u\in W_0^{1,p}(\Omega)\setminus \{0\}:\  D_u\Phi_{\lambda}(u)u=0 \right\}.
$$

Observe that all critical points of $\Phi_{\lambda}$ are contained in $\mathcal{N}_{\lambda}$. Moreover, consider the subsets $\mathcal{N}_{\lambda}^-,\mathcal{N}_{\lambda}^0,\mathcal{N}_{\lambda}^+\subset \mathcal{N}_{\lambda}$ defined by 

$$
\mathcal{N}_{\lambda}^-=\{u\in \mathcal{N}_{\lambda}:\ D_{uu}\Phi_\lambda (u)(u,u)<0\}.
$$
$$
\mathcal{N}_{\lambda}^0=\{u\in \mathcal{N}_{\lambda}:\ D_{uu}\Phi_\lambda (u)(u,u)=0\}.
$$
$$
\mathcal{N}_{\lambda}^+=\{u\in \mathcal{N}_{\lambda}:\ D_{uu}\Phi_\lambda (u)(u,u)>0\}.
$$

When $\mathcal{N}_{\lambda}^-,\mathcal{N}_{\lambda}^+\neq \emptyset$, it follows from the implicit function theorem that $\mathcal{N}_{\lambda}^-,\mathcal{N}_{\lambda}^+$ are $C^1$ manifolds of codimension one in  $W_0^{1,p}(\Omega)$.  Moreover, denoting $\mathcal{T}_u(\mathcal{N}_{\lambda}^-\cup \mathcal{N}_{\lambda}^+)$ as the tangent space of the manifold $\mathcal{N}_{\lambda}^-\cup \mathcal{N}_{\lambda}^+$ at the point $u$ we have the following results

\begin{prop}\label{constramin} Take $\lambda>0$ and $u\in \mathcal{N}_{\lambda}^-\cup \mathcal{N}_{\lambda}^+$. Then $D_u\Phi_{\lambda}(u)v=0$ for all $v\in W_0^{1,p}(\Omega)$ if and only if $D_u\Phi_{\lambda}(u)v=0$ for all $v\in \mathcal{T}_u(\mathcal{N}_{\lambda}^-\cup \mathcal{N}_{\lambda}^+)$. 
\end{prop}

\begin{cor}\label{solureg} Suppose that $\Phi$ restricted to $\mathcal{N}_{\lambda}^-\cup \mathcal{N}_{\lambda}^+$ has a critical point $u$, that is, $D_u\Phi_{\lambda}(u)v=0$ for all $v\in \mathcal{T}_u(\mathcal{N}_{\lambda}^-\cup \mathcal{N}_{\lambda}^+)$. Then, $u$ is a solution of \eqref{pq} and $u\in C^{1,\alpha}(\overline{\Omega})$ for some $\alpha\in (0,1)$.
\end{cor}

\begin{proof}
		 From the definition of weak solution and the Proposition \ref{constramin}, $u$ is a solution of \eqref{pq}. For the regularity, note from Tan-Fang \cite{tanfang} that $u\in L^\infty(\Omega)$ (one can also use Moser iteration), therefore from Tolksdorf and Lieberman \cite{tolks,lieberbound} the proof is completed.
\end{proof}

Now we consider the fibering approach (see Pohozaev \cite{poh}): let $\phi_{\lambda,u}:[0,\infty)\to \mathbb{R}$ be the real function defined by 
\begin{equation}\label{philambdau}
\phi_{\lambda,u}(t):=\Phi_\lambda(tu),
\end{equation}
 where $u\in W_0^{1,p}(\Omega)\setminus\{0\}$. The understanding of the fibering maps will be of extremely importance in the next sections.   

\begin{prop}\label{fibermaps} For each $u\in W_0^{1,p}(\Omega)\setminus \{0\}$ and $\lambda>0$, the function $\phi_{\lambda,u}$ is of class $C^\infty$ over the interval $(0,\infty)$. Moreover, if $F(u)\le 0$ then $\phi_{\lambda,u}$ has only one critical point at $t^+_\lambda(u)\in (0,\infty)$, which satisfies $\phi''_{\lambda,u}(t^+_\lambda(u))>0$. If $F(u)> 0$ then there are three possibilities
	
	\begin{description}
		\item[(I)] There are only two critical points for $\phi_{\lambda,u}$. One critical point at  $t_\lambda^+(u)$ with $\phi''_{\lambda,u}(t^+_\lambda(u))>0$ and the other one at $t^-(u)$ with $\phi''_{\lambda,u}(t^-_\lambda(u))<0$. Moreover $\phi_{\lambda,u}$ is decreasing over the intervals $[0,t_\lambda^+(u)]$, $[t_\lambda^-(u),\infty)$ and increasing over the interval $[t_\lambda^+(u),t_\lambda^-(u)]$  (evidently $0<t_\lambda^+(u)<t_\lambda^-(u)$). \\
		\item[(II)] There is only one critical point for $\phi_{\lambda,u}$, which is a saddle point at $t_\lambda^0(u)>0$. Moreover $\phi_{\lambda,u}$ is decreasing. \\
		\item[(III)] The function $\phi_{\lambda,u}$ is decreasing and has no critical points.
	\end{description}	
\end{prop}

\begin{proof}
	 The proof is straightforward.
\end{proof}

The following pictures give the possible graphs of the fiber maps. The case $F(u)\le 0$ corresponds to the Figure \ref{8}. The case $\mathbf{(I)}$ corresponds to Figure \ref{2}, the case $\mathbf{(II)}$ corresponds to Figure \ref{3} and the case $\mathbf{(III)}$ corresponds to Figure \ref{3}.

Observe that when $F(u)\le 0$, the graph of $\phi_{\lambda,u}$ will be always as in the Figure \ref{8} for any $\lambda>0$, however, when $F(u)>0$, this does not happen. Indeed, one can easily see that if $F(u)>0$ then, for $\lambda>0$ near $0$, we have the graph as in the Figure \ref{2}. By increasing $\lambda$, we can find some $\lambda(u)$ for which the graph of the fiber map will be as in the Figure \ref{3}. After $\lambda(u)$ the graph will be similar to \ref{4}.

\begin{figure}[H]
\centering\footnotesize   
{\def\svgwidth{0.28\linewidth}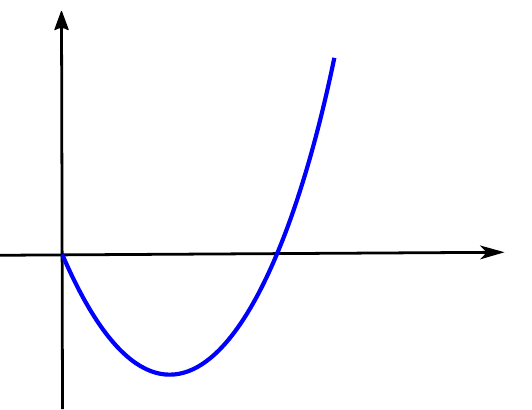}
\caption{Fiber map graph for $F(u)\le 0$}
\label{8}
\end{figure}

\begin{figure}[H]
	\begin{center}
		{ \footnotesize   
			\subfigure[][$F(u)>0$ \ $\mathbf{(I)}$]{\label{2}\def\svgwidth{0.35\linewidth}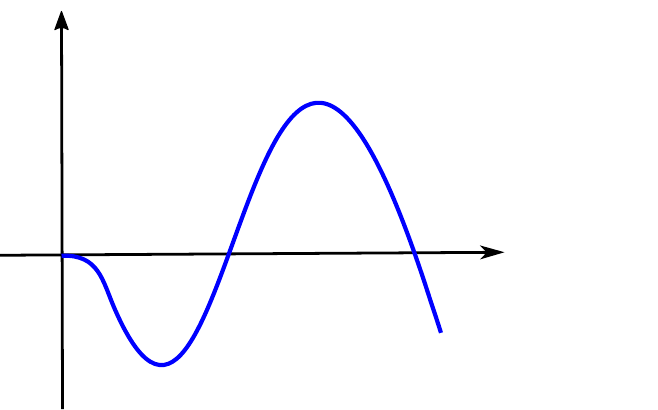} 
			\subfigure[][$F(u)>0$ \ $\mathbf{(II)}$]{\label{3}\def\svgwidth{0.28\linewidth}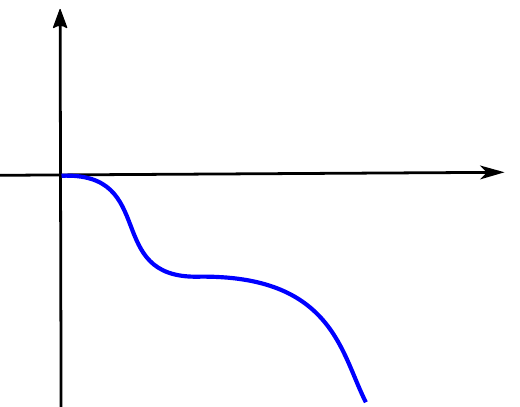}
			\subfigure[][$F(u)>0$ \ $\mathbf{(III)}$]{\label{4}\def\svgwidth{0.28\linewidth}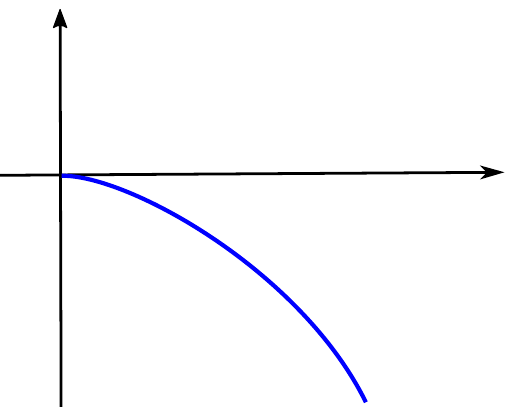}
		} \caption{Fiber map graphs for $F(u)>0$}
	\end{center}
\end{figure}

\begin{rem} If $f\ge 0$ then only $\mathbf{(I)}$, $\mathbf{(II)}$ and $\mathbf{(III)}$ may happen.
\end{rem}

From the previous discussion, one can see that for each $u\in W_0^{1,p}(\Omega)\setminus \{0\}$ with $F(u)>0$, there is a unique $\lambda=\lambda(u)>0$ such that $\phi_{\lambda,u}$ satisfies $\mathbf{(II)}$. Indeed, this is equivalent to solve the system (with respect to the variables $t$, $\lambda$)

\begin{equation*}\label{system}
\left\{
\begin{aligned}
\|tu\|^p-\lambda \|tu\|_q^q-F(tu)  &= 0, \\
p\|tu\|^p-\lambda q\|tu\|_q^q-\gamma F(tu)  &= 0.
\end{aligned}
\right.
\end{equation*}
It follows that 

\begin{equation}\label{rayleigh}
\left\{
\begin{aligned}
t(u)&= \left(\frac{p-q}{\gamma-q}\frac{\|u\|^p}{F(u)}\right)^{\frac{1}{\gamma-p}}, \\
\lambda(u) &\equiv \frac{\gamma-p}{\gamma-q}\left(\frac{p-q}{\gamma-q}\right)^{\frac{p-q}{\gamma-p}} \frac{\|u\|^{p\frac{\gamma-q}{\gamma-p}}}{\|u\|_q^qF(u)^{\frac{p-q}{\gamma-p}}}=\left(\frac{\gamma-p}{\gamma-q}\frac{\|u\|^p}{\|u\|_q^q}\right)\left(\frac{p-q}{\gamma-q}\frac{\|u\|^p}{F(u)}\right)^{\frac{p-q}{\gamma-p}}.
\end{aligned}
\right.
\end{equation}

From the construction we conclude that for each $u\in W_0^{1,p}(\Omega)\setminus\{0\}$ with $F(u)>0$ and $\lambda \in(0,\lambda(u))$ the fiber map $\phi_{\lambda,u}$ satisfies $\mathbf{(I)}$ while $\phi_{\lambda(u),u}$ satisfies $\mathbf{(II)}$ and $\phi_{\lambda,u}$ satisfies $\mathbf{(III)}$ for all $\lambda>\lambda(u)$. Moreover $\mathcal{N}_\lambda^0\neq \emptyset $ if and only if there exists $u\in W_0^{1,p}(\Omega)\setminus\{0\}$ such that $\lambda =\lambda(u)$. Observe that $t(u)=t_{\lambda(u)}^0(u)$. Define the extremal value (see Il'yasov \cite{ilyas1,ilyasENMM})

\begin{equation}\label{extremal}
\lambda^*\equiv \frac{\gamma-p}{\gamma-q}\left(\frac{p-q}{\gamma-q}\right)^{\frac{p-q}{\gamma-p}} \inf_{u\in W_0^{1,p}\setminus\{0\}} \left\{\frac{\|u\|^{p\frac{\gamma-q}{\gamma-p}}}{\|u\|_q^qF(u)^{\frac{p-q}{\gamma-p}}}:\ F(u)>0 \right\}.
\end{equation}

\begin{prop}\label{extre} The following holds true
	
	\begin{description}
		\item[(i)] the function $\lambda$, defined in \eqref{rayleigh}, is $0$-homogeneous and $0<\lambda^*<\infty$;
		\item[(ii)] $\mathcal{N}_{\lambda^*}^0\neq \emptyset$ and 
		
		$$
		\mathcal{N}_{\lambda^*}^0=\{u\in\mathcal{N}_{\lambda^*}:\ F(u)>0,\ \lambda(u)=\lambda^*\}.
		$$
		
		Moreover, each $u\in\mathcal{N}_{\lambda^*}^0$ satisfies  
		
		$$
		-p\Delta_p u-\lambda^* q|u|^{q-2}u-\gamma |u|^{\gamma-2}u=0;		
		$$
		\item[(iii)] $\mathcal{N}_\lambda^0=\emptyset$ for each $\lambda\in (0,\lambda^*)$ and $\mathcal{N}_\lambda^0\neq\emptyset$ for each $\lambda\in [\lambda^*,\infty)$.
	\end{description}
\end{prop}

\begin{proof} $\mathbf{(i)}$ The first part is obvious and the second is a consequence of the Sobolev embedding. 
	
	$\mathbf{(ii)}$ Since $\lambda$ is $0$-homogeneous, we have that 
	
	$$
	\lambda^*=\frac{\gamma-p}{\gamma-q}\left(\frac{p-q}{\gamma-q}\right)^{\frac{p-q}{\gamma-p}} \inf_{v\in S} \left\{\frac{1}{\|v\|_q^qF(v)^{\frac{p-q}{\gamma-p}}}:\ F(v)>0 \right\},
	$$
	where $S\equiv \{u\in W_0^{1,p}(\Omega):\ \|u\|=1\}$. Let $v_n\in S$ satisfies $F(v_n)>0$ and  $\lambda(v_n)\to \lambda^*$. Once $\|v_n\|=1$, we can assume that $v_n\rightharpoonup v$ in $W_0^{1,p}(\Omega)$ and $v_n\to v$ in $L^p(\Omega),L^\gamma(\Omega)$. Note that $v\not\equiv 0$ because if not then $\lambda(v_n)\to \infty$. It follows that $v/\|v\|\in S$ and $F(v/\|v\|)>0$. We claim that $v_n\to v$ in $W_0^{1,p}(\Omega)$. Indeed, if not, by the weak lower semi-continuity of the norm, we obtain that 
	
	$$
	\lambda\left(\frac{v}{\|v\|}\right)=\lambda(v)<\liminf \lambda(v_n)=\lambda^*,
	$$
	which is an absurd, therefore, $v_n\to v$ in $W_0^{1,p}(\Omega)$ and consequently $v\in S$, $F(v)>0$ and  $\lambda(v)=\lambda^*$. Therefore $t(v)v\in \mathcal{N}_{\lambda(v)=\lambda^*}^0$ and  $\mathcal{N}_{\lambda^*}^0\neq \emptyset$.  Once $\mathcal{N}_{\lambda^*}^0\neq \emptyset$, the equality 	
	$	\mathcal{N}_{\lambda^*}^0=\{u\in\mathcal{N}_{\lambda^*}:\ F(u)>0,\ \lambda(u)=\lambda^*\}$ is obvious.
	
	To prove that any $u\in\mathcal{N}_{\lambda^*}^0$ satisfies  
	
	$$
	-p\Delta_p u-\lambda^* q|u|^{q-2}u-\gamma |u|^{\gamma-2}u=0,		
	$$   we note that $D_u\lambda(u)w=0$ for all $w\in W_0^{1,p}(\Omega)$ and therefore
	
	\begin{equation}\label{ectresol}
	\begin{split}
	&\left(\frac{\gamma-q}{\gamma-p}\right)\|u\|_q^qF(u)^{\frac{p-q}{\gamma-p}}\|u\|^{p\frac{p-q}{\gamma-p}}(-p\Delta_p uw) \\ &-\|u\|^{p\frac{\gamma-q}{\gamma-p}}\left[\left(\frac{p-q}{\gamma-p}\right)\|u\|_q^qF(u)^{\frac{p-q-(\gamma-p)}{\gamma-p}}(\gamma |u|^{\gamma-2}uw)+F(u)^{\frac{p-q}{\gamma-p}}(q |u|^{q-2}uw)\right]=0.
	\end{split}
	\end{equation}
	
	From \eqref{ectresol} we conclude that 
	
	\begin{equation}\label{ect1}
	-p\Delta_p uw-\left(\frac{\gamma-p}{\gamma-q}\right)\frac{\|u\|^p}{\|u\|_q^q}q |u|^{q-2}uw-\left(\frac{p-q}{\gamma-q}\right)\frac{\|u\|^p}{F(u)}\gamma |u|^{\gamma-2}uw=0,\ \forall\ w\in W_0^{1,p}(\Omega).
	\end{equation}
	
	Once $u\in \mathcal{N}_{\lambda^*}^0$, we have that 
	
	\begin{equation}\label{ect2}
	\left(\frac{\gamma-p}{\gamma-q}\right)\frac{\|u\|^p}{\|u\|_q^q}=\lambda^*,\ \ \ \left(\frac{p-q}{\gamma-q}\right)\frac{\|u\|^p}{F(u)}=1.
	\end{equation}
	
	From \eqref{ect1} and \eqref{ect2} we infer that	$u$ satisfies 
	
	$$
	-p\Delta_p u-\lambda^* q|u|^{q-2}u-\gamma f|u|^{\gamma-2}u=0.		
	$$ 
	
	$\mathbf{(iii)}$ it is a consequence of the definition of $\lambda^*$.
	
\end{proof}

The following results about the Nehari set $\mathcal{N}_{\lambda^*}^0$ will be essential to prove the existence of solutions for $\lambda\ge \lambda^*$.

\begin{cor}\label{compact} The set $\mathcal{N}_{\lambda^*}^0$ is compact.
\end{cor}

\begin{proof}
	 
	First, observe that $u\in \mathcal{N}_{\lambda^*}^0$ implies
	
	$$
	\|u\|^p-\lambda^* \|u\|_q^q-F(u)=0=p\|u\|^p-\lambda^* q\|u\|_q^q-\gamma F(u).
	$$
	
	It follows that there exist positive constants $c,C$ such that 
	
	\begin{equation}\label{boundeextreneha}
	c\le \|u\|\le C|\lambda^*|^{\frac{1}{p-q}},\ \forall\ u\in \mathcal{N}_{\lambda^*}^0.
	\end{equation}
	
	Let $u_n\in \mathcal{N}_{\lambda^*}^0$ for $n=1,2,\ldots$. From the Proposition \ref{extre} we know that 
	
	\begin{equation}\label{boundeextreneha1}
	-p\Delta_p u_n-\lambda^* q|u_n|^{q-2}u_n-\gamma f |u_n|^{\gamma-2}u_n=0,\ \forall\ n=1,2,\ldots 
	\end{equation}
	
	From \eqref{boundeextreneha} we can assume that, up to a subsequence, $u_n \rightharpoonup u$ in $W_0^{1,p}(\Omega)$ and $u_n\to u$ in $L^p(\Omega),L^\gamma(\Omega)$. From \eqref{boundeextreneha1} and the $S^+$ property of the $p$-Laplacian operator (see Dr\'abek-Milota \cite{drabekMilota}) we conclude that $u_n\to u$ in	$W_0^{1,p}(\Omega)$ and consequently $\mathcal{N}_{\lambda^*}^0$ is compact.
\end{proof}

For $\lambda>0$ we define
\begin{equation*}\label{N^}
\hat{\mathcal{N}}_\lambda= \{u\in W_0^{1,p}(\Omega)\setminus\{0\}:\ F(u)>0,\ \phi_{\lambda,u}\ \mbox{satisfies}\ \mathbf{(I)}\},
\end{equation*}
and
\begin{equation*}\label{N^+}
\hat{\mathcal{N}}_\lambda^+= \{u\in W_0^{1,p}(\Omega)\setminus\{0\}:\ F(u)\le 0\}.
\end{equation*}

\begin{rem}\label{rrmk2}
	Note that for $\lambda>0$ we have $\hat{\mathcal{N}}_\lambda\neq \emptyset$. Moreover, for $\lambda_1,\lambda_2\in (0,\lambda^*)$ we also have that $\hat{\mathcal{N}}_{\lambda_1}=\hat{\mathcal{N}}_{\lambda_2}$ and $\hat{\mathcal{N}}^+_{\lambda_1}=\hat{\mathcal{N}}^+_{\lambda_2}$. 
\end{rem}

\begin{rem}\label{rrmk3}
One can easily see that if $u\in \hat{\mathcal{N}}_\lambda\cup \hat{\mathcal{N}}_\lambda^+$ then $tu\in  \hat{\mathcal{N}}_\lambda\cup \hat{\mathcal{N}}_\lambda^+$ for all $t>0$. It follows that $\hat{\mathcal{N}}_\lambda\cup \hat{\mathcal{N}}_\lambda^+$ is the positive cone generated by the Nehari manifold $\mathcal{N}_\lambda^+\cup \mathcal{N}_\lambda^-$, that is 

$$
\hat{\mathcal{N}}_\lambda\cup \hat{\mathcal{N}}_\lambda^+=\{tu:\ t>0,\ u\in \mathcal{N}_\lambda^+\cup \mathcal{N}_\lambda^-\}.
$$
\end{rem}

Let $\overline{\hat{\mathcal{N}}_\lambda\cup \hat{\mathcal{N}}_\lambda^+}$ denotes the closure of $\hat{\mathcal{N}}_\lambda\cup \hat{\mathcal{N}}_\lambda^+$ with respect to the norm topology.

\begin{prop}\label{closure} There holds
	
	$$
	\overline{\hat{\mathcal{N}}_{\lambda^*}\cup \hat{\mathcal{N}}_{\lambda^*}^+}=\hat{\mathcal{N}}_{\lambda^*}\cup \hat{\mathcal{N}}_{\lambda^*}^+\cup\{tu:\ t>0,\ u\in \mathcal{N}_{\lambda^*}^0\}\cup\{0\}.
	$$
	
\end{prop}

\begin{proof} Let us first show that $\overline{\mathcal{N}_\lambda^+\cup \mathcal{N}_\lambda^-}= \mathcal{N}_\lambda^+\cup \mathcal{N}_\lambda^-\cup\mathcal{N}_{\lambda^*}^0\cup\{0\}$. 
\vskip.3cm
	
{\bf Case 1:} 	$u_n \in \mathcal{N}_\lambda^-$ satisfies $u_n \to u$ in $W_0^{1,p}(\Omega)$.
\vskip.3cm	
	
We have that 

\begin{equation}\label{test1}
\left\{
\begin{aligned}
\|u_n\|^p-\lambda^* \|u_n\|_q^q- F(u_n)&= 0, \\
p\|u_n\|^p-\lambda^* q\|u_n\|_q^q-\gamma  F(u_n)& < 0,
\end{aligned}\right.\ \forall\ n=1,2,\cdots
\end{equation}	

From \eqref{test1} one can easily see that if $F(u)\neq 0$ then, $u\in \mathcal{N}_{\lambda^*}^-\cup\mathcal{N}_{\lambda^*}^0$, while if $F(u)=0$ then $u=0$.

\vskip.3cm

{\bf Case 2:}	$u_n \in \mathcal{N}_\lambda^+$ satisfies $u_n \to u$ in $W_0^{1,p}(\Omega)$.
\vskip.3cm

We have 

\begin{equation}\label{test2}
\left\{
\begin{aligned}
\|u_n\|^p-\lambda^* \|u_n\|_q^q- F(u_n)&= 0, \\
p\|u_n\|^p-\lambda^* q\|u_n\|_q^q-\gamma  F(u_n)& > 0,
\end{aligned}\right.\ \forall\ n=1,2,\cdots
\end{equation}	

From \eqref{test2} one can easily see that if $F(u)\neq 0$ then, $u\in  \mathcal{N}_{\lambda^*}^+\cup\mathcal{N}_{\lambda^*}^0$, while if $F(u)=0$ then $u=0$ or $u\in \mathcal{N}_{\lambda^*}^+$.

It follows that $\overline{\mathcal{N}_\lambda^+\cup \mathcal{N}_\lambda^-}=  \mathcal{N}_\lambda^+\cup \mathcal{N}_\lambda^-\cup\mathcal{N}_{\lambda^*}^0\cup\{0\}$ and from the Remark \ref{rrmk3} the proof is completed.

\end{proof}

Define $t_{\lambda^*}:\overline{\hat{\mathcal{N}}_\lambda}\setminus \{0\}\to \mathbb{R}$ and $s_{\lambda^*}:W_0^{1,p}(\Omega)\setminus \{0\}\to \mathbb{R}$ by 
\begin{equation}\label{t-lambda*}
t_{\lambda^*}(w)=
\begin{cases}
 t_{\lambda^*}^-(w), & \mbox{if} \quad w\in \hat{\mathcal{N}}_{\lambda^*}\\
 t_{\lambda^*}^0(w), & \mbox{otherwise},
\end{cases}
\end{equation}
and
\begin{equation}\label{s-lambda*}
s_{\lambda^*}(u)=
\begin{cases}
t_{\lambda^*}^+(u), & \mbox{if} \quad u\in \hat{\mathcal{N}}_{\lambda^*}\cup \hat{\mathcal{N}}_{\lambda^*}^+\\
t_{\lambda^*}^0(u), & \mbox{otherwise}.
\end{cases}
\end{equation} 
Let $S\equiv \{u\in W_0^{1,p}(\Omega):\ \|u\|=1\}$.
\begin{prop}\label{homeo} There holds
	
	\begin{description}
		\item[(i)] $t_{\lambda^*}$ is a continuous function. Moreover, the function $P^-:S\cap\overline{\hat{\mathcal{N}}_\lambda}\to \mathcal{N}_{\lambda^*}^-\cup \mathcal{N}_{\lambda^*}^0$ defined by $P^-(v)=t_{\lambda^*}(v)v$ is a homeomorphism;
		\item[(ii)] $s_{\lambda^*}$ is a continuous function. Moreover, the function $P^+:S\to \mathcal{N}_{\lambda^*}^+\cup \mathcal{N}_{\lambda^*}^0$ defined by $P^+(v)=s_{\lambda^*}(v)v$ is a homeomorphism.
	\end{description}   
\end{prop}

\begin{proof}
	$\mathbf{(i)}$ The continuity follows from the inequalities

	\begin{equation*}
	\left\{
	\begin{aligned}
	t_{\lambda^*}(v)^p\|v\|^p-\lambda^*t_{\lambda^*}(v)^q\|v\|_q^q-t_{\lambda^*}(v)^\gamma F(v)&= 0, \\
	pt_{\lambda^*}(v)^p\|v\|^p-\lambda^*qt_{\lambda^*}(v)^q\|v\|_q^q-\gamma t_{\lambda^*}(v)^\gamma F(v)& \le 0.
	\end{aligned}\right.
	\end{equation*}

	To prove that $P^-$ is a homeomorphism, observe that the continuous function $(P^-)^{-1}:	\mathcal{N}_{\lambda^*}^-\cup \mathcal{N}_{\lambda^*}^0\to S\cap\overline{\hat{\mathcal{N}}_\lambda}$ defined by $(P^-)^{-1}(u)=u/\|u\|$ is the inverse of $P^-$ .
	
	$\mathbf{(ii)}$ Similar to $\mathbf{(i)}$.
	
\end{proof}

\begin{cor}\label{emptyinterior} Consider $\mathcal{N}_{\lambda^*}\subset W_0^{1,p}(\Omega)$ with its topology induced by the norm of $ W_0^{1,p}(\Omega)$.  Then, the set $\mathcal{N}_{\lambda^*}^0\subset \mathcal{N}_{\lambda^*}$ has empty interior.
\end{cor}

\begin{proof}
Suppose on the contrary that for some $v\in \mathcal{N}_{\lambda^*}^0$ there is an open neighborhood $\mathcal{U}\subset \mathcal{N}_{\lambda^*}^0$ of $v$. Define
	
	$$
	P(\mathcal{U})=\left\{\frac{u}{\|u\|}:\ u\in \mathcal{U}  \right\}.
	$$
	
	From the Proposition \ref{homeo} follows that $P(\mathcal{U})\subset S$ is an open neighborhood of $v/\|v\|$ on the sphere. Once $P(\mathcal{U})$ is an open set of the sphere its closure over the sphere is not compact, however, this is an absurd because it would imply that the closure of $\mathcal{U}$ is not compact, which contradicts the Corollary \ref{compact}.
\end{proof}

From now on, for $\lambda>0$, let $J_\lambda^-:\hat{\mathcal{N}}_\lambda\to \mathbb{R}$ and $J_\lambda^+:\hat{\mathcal{N}}_\lambda\cup \hat{\mathcal{N}}_\lambda^+\to \mathbb{R}$ be defined by 
\begin{equation*}\label{jlambda}
J_{\lambda}^-(u)=\Phi_\lambda(t_\lambda^-(u)u),
\quad
\mbox{and}
\quad
J_\lambda^+(u)=\Phi_\lambda(t_\lambda^+(u)u).
\end{equation*}
We consider the following constrained minimization problems
\begin{equation*}\label{j mp}
\hat{J}_\lambda^{-}=\inf \{J_\lambda^{-}(u):\ u\in\mathcal{N}_\lambda^-\}\quad\mbox{ and } \quad \hat{J}_\lambda^{+}=\inf \{J_\lambda^{+}(u):\ u\in\mathcal{N}_\lambda^+\}.
\end{equation*}

\begin{rem}\label{jcrifun}
	 Observe that $J_\lambda^-,J_\lambda^+$ are $0$-homogeneous functionals. Moreover, from the implicit function theorem they are $C^1$ functionals and from the Proposition \ref{constramin} any minimizer of $\hat{J}_\lambda^{-}$ or $\hat{J}_\lambda^{+}$ is a critical point for $\Phi_\lambda$.
\end{rem}

 To simplify, when possible we will use the symbols $\hat{J}_\lambda^\mp$, $t_\lambda^\mp$ and so on to indicate $\hat{J}_\lambda^-$, $t_\lambda^-$, $\hat{J}_\lambda^+$, $t_\lambda^+$.  For the next sections, we will be interested in minimizing the functionals $J_\lambda^\mp$.

\begin{prop}\label{decrea1} Take $v\in W_0^{1,p}(\Omega)\setminus\{0\}$. Let $I\subset \mathbb{R}$ be an open interval such that $t_\lambda^\mp(v)$ are well defined for all $\lambda\in I$. There holds 
	
	\begin{description}
		\item[(i)] the functions $I\ni \lambda \mapsto t_\lambda^{\mp}(v)$ is $C^1$. Moreover, $I\ni \lambda \mapsto t_\lambda^{-}(v)$ is decreasing while $I\ni \lambda \mapsto t_\lambda^{+}(v)$ is increasing;
		\item[(ii)] the functions  $I\ni \lambda \mapsto J_\lambda^{\mp}(v)$ are continuous and decreasing.
	\end{description}
\end{prop}

\begin{proof} $\mathbf{(i)}$ Once $t_\lambda^{\mp}(v)v\in \mathcal{N}_\lambda^{\mp}$, we have from the implicit function theorem that $I\ni \lambda \mapsto t_\lambda^{\mp}(v)$ are $C^1$ and  
	
	$$
	\frac{\partial }{\partial \lambda}t_\lambda^{\mp}(v)=\frac{t_\lambda^{\mp}(v)\|t_\lambda^{\mp}(v)v\|_q^q}{p\|t_\lambda^{\mp}(v)v\|^p-q\lambda \|t_\lambda^{\mp}(v)v\|_q^q-\gamma  F(t_\lambda^{\mp}(v) v)}, \ \forall \lambda\in I.
	$$
	
	Therefore 	$	\frac{\partial }{\partial \lambda}t_\lambda^{-}(v)<0$ and $	\frac{\partial }{\partial \lambda}t_\lambda^{+}(v)>0$ for $\lambda\in I$.
	
	$\mathbf{(ii)}$ Indeed, from $\mathbf{(i)}$ we have that 
	
	$$
	\frac{\partial }{\partial \lambda} J_\lambda^{\mp}(v)=-\frac{\|t_\lambda^{\mp}(v)v\|_q^q}{q}<0.
	$$

\end{proof}

Fix some $w,u\in W_0^{1,p}(\Omega)\setminus \{0\}$, $\lambda'\in (0,\lambda^*)$ and suppose that $w\in \hat{\mathcal{N}}_{\lambda'}$, $u\in \hat{\mathcal{N}}_{\lambda'}\cup \hat{\mathcal{N}}_{\lambda'}^+$. Observe from the Remark \ref{rrmk2} that $t_\lambda^-(w)$ and $t_\lambda^+(u)$ are well defined for all $\lambda\in (0,\lambda^*)$. From the Proposition \ref{decrea1}, we obtain that 

\begin{cor}\label{decrea} If $w\in \hat{\mathcal{N}}_{\lambda'}$, $u\in \hat{\mathcal{N}}_{\lambda'}\cup \hat{\mathcal{N}}_{\lambda'}^+$ for some $\lambda'\in (0,\lambda^*)$ then 
	
	\begin{description}
		\item[(i)] the functions $(0,\lambda^*)\ni \lambda \mapsto t_\lambda^{-}(w)$, $(0,\lambda^*)\ni \lambda \mapsto t_\lambda^{+}(u)$ are $C^1$. Moreover, $(0,\lambda^*)\ni \lambda \mapsto t_\lambda^{-}(w)$ is decreasing while $(0,\lambda^*)\ni \lambda \mapsto t_\lambda^{+}(u)$ is increasing;
		\item[(ii)] the functions  $(0,\lambda^*)\ni \lambda \mapsto J_\lambda^{-}(w)$, $(0,\lambda^*)\ni \lambda \mapsto J_\lambda^{+}(u)$ are continuous and decreasing.	
	\end{description}
\end{cor}

In the next Corollary we study the behavior of the fiber maps when $\lambda\uparrow \lambda^*$ (see Figure \ref{14}). 

\begin{figure}[H]
	\centering \tiny   
	{\def\svgwidth{0.5\linewidth}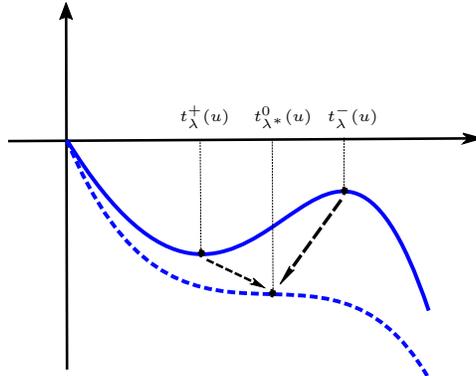}
	\caption{Behavior of the fiber maps accordingly with the parameter $\lambda$}
	\label{14}
\end{figure}

\begin{cor}\label{convess} 	Suppose that $u\notin \hat{\mathcal{N}}_{\lambda^*}^+$. Then 
	$$
	\lim_{\lambda\uparrow \lambda^*}t_\lambda^{-}(u)=t_{\lambda^*}(u),\ \ \lim_{\lambda\uparrow \lambda^*}t_\lambda^{+}(u)=s_{\lambda^*}(u)
	$$
	and
	$$
	\lim_{\lambda\uparrow \lambda^*}J_\lambda^-(u)=\Phi_{\lambda^*}(t_{\lambda^*}(u)u),\ \ \lim_{\lambda\uparrow \lambda^*}J_\lambda^+(u)=\Phi_{\lambda^*}(s_{\lambda^*}(u)u),
	$$
	with $t_{\lambda^*}(u)$ and $s_{\lambda^*}(u)$ defined as in \eqref{t-lambda*} and \eqref{s-lambda*}.
\end{cor}

\begin{proof} If $u\in \hat{\mathcal{N}}_{\lambda^*}$ the proof follows from the Proposition \ref{decrea1}. If $u\notin \hat{\mathcal{N}}_{\lambda^*}\cup \hat{\mathcal{N}}_{\lambda^*}^+$ then, from the definition of $\lambda^*$, we have that $u\in \hat{\mathcal{N}}_\lambda$ for all $\lambda\in (0,\lambda^*)$ and $\lambda^*=\lambda(u)$. Moreover
	
	\begin{equation*}
	\left\{
	\begin{aligned}
	\|t_{\lambda}^-(u)u\|^p-\lambda \|t_{\lambda}^-(u)u\|_q^q- F(t_{\lambda}^-(u)u)&= 0, \\
	p\|t_{\lambda}^-(u)u\|^p-\lambda q\|t_{\lambda}^-(u)u\|_q^q-\gamma  F(t_{\lambda}^-(u)u)& < 0,
	\end{aligned}\right.\ \forall\ \lambda\in (0,\lambda^*),
	\end{equation*}	
	and
	\begin{equation*}
	\left\{
	\begin{aligned}
	\|t_{\lambda}^+(u)u\|^p-\lambda \|t_{\lambda}^+(u)u\|_q^q- F(t_{\lambda}^+(u)u)&= 0, \\
	p\|t_{\lambda}^+(u)u\|^p-\lambda q\|t_{\lambda}^+(u)u\|_q^q-\gamma  F(t_{\lambda}^+(u)u)& > 0,
	\end{aligned}\right.\ \forall\ \lambda\in (0,\lambda^*),
	\end{equation*}	
	
	From the Corollary \ref{decrea} we can assume without loss of generality that $t_\lambda^-(u)\to t^-$, $t_\lambda^+(u)\to t^+$ as $\lambda\uparrow \lambda^*$ where $0<t^+\le t^-<\infty$. It follows that              
	
	\begin{equation}\label{tt}
	\left\{
	\begin{aligned}
	\|t^-u\|^p-\lambda^* \|t^-u\|_q^q- F(t^-u)&= 0, \\
	p\|t^-u\|^p-\lambda^* q\|t^-u\|_q^q-\gamma  F(t^-u)& \le 0,
	\end{aligned}\right.
	\end{equation}	
	and
	\begin{equation}\label{tt1}
	\left\{
	\begin{aligned}
	\|t^+u\|^p-\lambda^* \|t^+u\|_q^q- F(t^+u)&= 0, \\
	p\|t^+u\|^p-\lambda^* q\|t^+u\|_q^q-\gamma  F(t^+u)& \ge  0.
	\end{aligned}\right.
	\end{equation}

	We claim that $t^-=t^+$. Indeed, suppose  on the contrary that $t^-<t^+$. It follows from \eqref{tt} and \eqref{tt1} that $t^-=t_{\lambda^*}^-(u)$ and $t^+=t_{\lambda^*}^+(u)$, for $t_{\lambda^*}(u)$ defined as in \eqref{t-lambda*}, however this contradicts the fact that $\lambda(u)=\lambda^*$ and the Proposition \ref{extre}, therefore $t^-=t^+$ and from \eqref{tt}, \eqref{tt1} we conclude that $t^-=t^+=t_{\lambda^*}^0(u)$. The second limit is straightforward.
\end{proof}

%%%%%%%%%%%%%%%%%%%%%%%%%%%%%%%%%%%%%%%%%%%%%%%%%%%%%%%%%%%%%%%%%%%%%%%%%%%%%%%%%%%%%%%%%%%%%%%%%%%%%%%

\section{Existence of solutions in $[0,\lambda^*]$}

In this section we show existence of positive solutions to the problem \eqref{pq} for $\lambda\in [0,\lambda^*]$. Some of the ideas used here can be found in \cite{ilyas1,brownwu,ilyasENMM}.

\begin{lem}\label{positivesol1} For each $\lambda\in [0,\lambda^*]$, there exists $0<w_\lambda\in \mathcal{N}_\lambda^-$ and $0<u_\lambda\in \mathcal{N}_\lambda^+$ solutions of \eqref{pq}. Moreover $w_\lambda,u_\lambda\in C^{1,\alpha}(\overline{\Omega})$ for some $\alpha\in (0,1)$.
\end{lem}

The proof will be given at the end of this section. 

\begin{prop}
	\label{jproper}
	Let $\lambda>0$. The functional $\Phi_\lambda$ is weakly lower semi-continuous. Moreover, the functionals $J_\lambda^\mp$ are coercive.
\end{prop}

\begin{proof}
	That $\Phi_{\lambda}$ is  weakly lower semi-continuous is a straightforward calculation. To prove coerciveness, note that for all $u\in \mathcal{N}_\lambda$ there holds 
	
	\begin{equation}\label{coer}
	\Phi_{\lambda}(u)\ge \left(\frac{1}{p}-\frac{1}{\gamma}\right)\int |\nabla u|^p-\left(\frac{1}{q}-\frac{1}{\gamma}\right)\lambda \int | u|^q,
	\end{equation}	
	which implies from the Sobolev embedding that $\Phi_{\lambda}$ is coercive over  $\mathcal{N}_\lambda$ and therefore $J_\lambda^{\mp}$ are coercive.	
\end{proof}

The next result is essential in proving that minimizing sequences does not converge weakly to zero. 

\begin{prop}\label{inequal} Suppose that $\mathcal{N}_\lambda^\mp\neq \emptyset$. Then
	\begin{description}
		\item[(i)] for each $u\in \mathcal{N}_\lambda^-$, there holds 
		
		\begin{equation*}
		(p-q)\|u\|^p<(\gamma-p) F(u);
		\end{equation*}
		\item[(ii)] for each $u\in \mathcal{N}_\lambda^+$, there holds
		
		\begin{equation*}
		(\gamma-p)\|u\|^p<\lambda (\gamma-q)\|u\|_q^q.
		\end{equation*}
		
	\end{description}
\end{prop}

\begin{proof} The proof is straightforward from the definitions.
	
\end{proof}

From the Proposition \ref{inequal} and the Sobolev embeddings we obtain

\begin{cor}\label{estimatives} There are constants $C_1,C_2>0$ such that
	
	\begin{description} 
		\item[(i)] for each $u\in \mathcal{N}_\lambda^-$, there holds 
		$$
		\|u\|>C_1\left(\frac{p-q}{\gamma-q}\right)^{\frac{1}{\gamma-p}};
		$$
		\item[(ii)] for each $u\in \mathcal{N}_\lambda^+$, there holds
		$$
		\|u\|<C_2\left(\frac{\gamma-q}{\gamma-p}\right)^{\frac{1}{p-q}}\lambda^{\frac{1}{p-q}}. 
		$$
	\end{description}	
\end{cor}

For each $\lambda>0$, we consider the following constrained minimization problems

$$
\hat{J}_\lambda^{\mp}=\inf \{J_\lambda^{\mp}(u):\ u\in\mathcal{N}_\lambda^\mp\}.
$$
Observe from the Proposition \ref{estimatives} that $\hat{J}_\lambda^{\mp}>-\infty$.

\begin{prop}\label{fnotcon} For each $\lambda>0$ there holds	
	
	\begin{description} 	
		\item[(i)] if $w_n\in \mathcal{N}_\lambda^-$ is a minimizing sequence for $\hat{J}_\lambda^-$ then there exists constants $c,C>0$ such that $c<\|w_n\|<C$;
		\item[(ii)] if $u_n\in \mathcal{N}_\lambda^+$ is a minimizing sequence for $\hat{J}_\lambda^+$ then there exists constants $c,C>0$ such that $c<\|u_n\|<C$.
	\end{description}
\end{prop}

\begin{proof}
	$\mathbf{(i)}$ Suppose that $w_n\in \mathcal{N}_\lambda^-$ satisfies $J^-_\lambda(w_n)\to \hat{J}_\lambda^-$. From the Corollary \ref{estimatives}, we only have to find $C$. However, from the Proposition \ref{jproper}, if $\|u_n\|\to \infty$ then we  conclude that $J_\lambda^-(u_n)\to \infty$ which contradicts the definition of $\hat{J}_\lambda^-$.
	
	$\mathbf{(ii)}$ Suppose that $u_n\in \mathcal{N}_\lambda^+$ satisfies $J^+_\lambda(u_n)\to \hat{J}_\lambda^+$. From the  Corollary \ref{estimatives}, we only have to find $c$. However, from the  Proposition \ref{jproper}, if $\|u_n\|\to 0$ the we conclude that $\hat{J}_\lambda^+\ge  0$ which is an absurd because $\hat{J}_\lambda^+<0$.	
	
\end{proof}

\begin{lem}\label{minglob}
	For each $\lambda\in (0,\lambda^*)$ there are two positive functions $w_\lambda\in \mathcal{N}_\lambda^-$ and  $u_\lambda\in \mathcal{N}_\lambda^-$ such that $J_\lambda^-(w_\lambda)=\hat{J}_\lambda^-$ and $J_\lambda^+(v_\lambda)=\hat{J}_\lambda^+$. 
\end{lem}

\begin{proof}
	We start with $\hat{J}_\lambda^-$. Suppose that $w_n\in \mathcal{N}_\lambda^-$ satisfies $J_\lambda^-(w_n)\to \hat{J}_\lambda^-$. From the Proposition \ref{fnotcon}, we may assume that $w_n\rightharpoonup w$ in $W_0^{1,p}(\Omega)$, $w_n\to w$ in $L^q(\Omega),L^\gamma(\Omega)$. Let us prove that $w\neq 0$ and $F(w)>0$. Indeed, if not, from the Proposition \ref{inequal} we would have that $\|w_n\|\to 0$, which contradicts the Proposition \ref{fnotcon}. Therefore $w\neq 0$ and $F(w)>0$.

	We claim that $w_n\to w$ in $W_0^{1,p}(\Omega)$. In fact, on the contrary, we would have that $\|w\|<\liminf \|w_n\|$ and thus 
	
	$$
	\liminf_{n\to \infty} D_u\Phi_\lambda\left(t_\lambda^-(w)w_n\right)>  D_u\Phi_{\lambda}\left(t_\lambda^-(w)w\right)=0,
	$$
	which implies from the Proposition \ref{fibermaps} that for sufficiently large $n$, $ D_u\Phi_\lambda\left(t_\lambda^-\left(w\right)w_n\right)>0$. Therefore, for sufficiently large $n$ we have that $t^+_\lambda(w_n)<t_\lambda^-\left(w\right)<t_\lambda^-(w_n)=1$ and hence

	\begin{align*}
	J_\lambda^-\left(w\right)  =\Phi_{\lambda}(t_\lambda^-(w)w)   
	<  \liminf_{n\to \infty} \Phi_\lambda\left(t_\lambda^-\left(w\right)w_n\right) 
	< \liminf_{n\to \infty} \Phi_\lambda\left(w_n\right)=\hat{J}^-_\lambda,                   
	\end{align*}
	which is a contradiction. Therefore $w_n\to w$ in $W_0^{1,p}(\Omega)$, $w\in \mathcal{N}_\lambda^-$ and $J_\lambda^-(w)=\hat{J}^-_\lambda$. 
	
	Now suppose that $u_n\in \mathcal{N}_\lambda^+$ satisfies $J_\lambda^+(u_n)\to \hat{J}_\lambda^+$.  From the Proposition \ref{fnotcon}, we may assume that $u_n\rightharpoonup u$ in $W_0^{1,p}(\Omega)$, $u_n\to u$ in $L^q(\Omega),L^\gamma(\Omega)$.  Let us prove that $u\neq 0$. Indeed, if not, from the Proposition \ref{inequal} we would have that $\|u_n\|\to 0$, which contradicts the Proposition \ref{fnotcon} We claim that $u_n\to u$ in $W_0^{1,p}(\Omega)$. In fact, on the contrary, we would have that $\|u\|<\liminf \|u_n\|$ and thus 
	$$
	\liminf_{n\to \infty} D_u\Phi_{\lambda}\left(t_{\lambda}^+(u)u_n\right)>  D_u\Phi_{\lambda}\left(t_\lambda^+\left(u\right)u\right)=0,
	$$
	which implies from the Proposition \ref{fibermaps} that for sufficiently large $n$, $ D_u\Phi_{\lambda}\left(t_{\lambda}^+\left(u\right)u_n\right)>0$. Therefore, for sufficiently large $n$ we have that $1=t^+_{\lambda}(u_n)<t_{\lambda}^+\left(u\right)$. It follows that $\Phi_\lambda\left(t_\lambda^+\left(u\right)u\right)<\Phi_\lambda\left(u\right)$ for sufficiently large $n$, and consequently	
	\begin{align*}
	J_\lambda^+(u)	=\Phi_\lambda\left(t_\lambda^+\left(u\right)u\right) 
	< \liminf_{n\to \infty}\Phi_\lambda(u_n) =\hat{J}^+_\lambda.                   
	\end{align*}
	which is an absurd. Therefore $u_n\to u$ in $W_0^{1,p}(\Omega)$, $u\in \mathcal{N}_\lambda^+$ and  $J_\lambda^+(u)=\hat{J}^+_\lambda$.	
\end{proof}

Now we study the problems $\hat{J}_{\lambda^*}^\mp$. First, observe from the Proposition \ref{homeo} and the Corollary \ref{emptyinterior} that if 

\begin{equation*}\label{hatphi}
\hat{\Phi}_{\lambda^*}^-=\inf\{ \Phi_{\lambda^*}(t_{\lambda^*}(w)w):\ w\in \mathcal{N}_{\lambda^*}^-\cup \mathcal{N}_{\lambda^*}^0 \},
\end{equation*}
and
$$
\hat{\Phi}_{\lambda^*}^+=\inf\{ \Phi_{\lambda^*}(s_{\lambda^*}(u)u):\ u\in \mathcal{N}_{\lambda^*}^+\cup \mathcal{N}_{\lambda^*}^0 \},
$$
with $t_{\lambda^*}(u)$ and $s_{\lambda^*}(u)$ defined as in \eqref{t-lambda*} and \eqref{s-lambda*}, then $\hat{J}_{\lambda^*}^\mp=\hat{\Phi}_{\lambda^*}^\mp$.  
\begin{prop}\label{contiex} There holds	
	\begin{description}
		\item[(i)] 	The functions $(0,\lambda^*]\ni\lambda\mapsto \hat{J}_\lambda^{\mp}$ are decreasing;
		\item[(ii)] 	$$
		\lim _{\lambda\uparrow \lambda^*}\hat{J}_\lambda^\mp=\hat{J}_{\lambda^*}^\mp.
		$$
	\end{description}
\end{prop}

\begin{proof} $\mathbf{(i)}$ Indeed, if $0<\lambda<\lambda'<\lambda^*$, we have from the Corollary \ref{decrea} item $\mathbf{(ii)}$ that  
	
	$$
	\hat{J}_{\lambda'}^-\le J_{\lambda'}^-(w_\lambda)<J_{\lambda}^-(w_\lambda)=\hat{J}_\lambda^-.
	$$
	
	Moreover, if $\lambda\in (0,\lambda^*)$ then from the Corollaries \ref{decrea} and \ref{convess} we obtain that  $\hat{J}_{\lambda^*}^-=\hat{\Phi}_{\lambda^*}^-\le \Phi_{\lambda^*}(t_{\lambda^*}(w)w)=\lim_{\overline{\lambda}\downarrow \lambda^*}\Phi_{\overline{\lambda}}(t^-_{\overline{\lambda}}(w)w)< J_\lambda^-(w)$, with $t_{\lambda^*}(u)$ defined as in \eqref{t-lambda*}, for all $w\in \mathcal{N}_{\lambda^*}^-\cup \mathcal{N}_{\lambda^*}^0$ and hence  $\hat{J}_{\lambda^*}^-\le \hat{J}_{\lambda}^-$.
	
	The same holds true for $\hat{J}_\lambda^-$. 
	
	$\mathbf{(ii)}$ Let $\lambda_n\uparrow \lambda^*$. From $\mathbf{(i)}$ we can assume that $\hat{J}_{\lambda_n}^-\to J\ge \hat{J}^-_{\lambda^*}$. Given $\delta>0$, suppose on the contrary that $J-\hat{J}_{\lambda^*}^-\ge \delta$. Fix $0<\delta'$ such that $2\delta '<\delta$ and choose $w_{\delta'}\in \mathcal{N}_{\lambda^*}^-$ such that $J_{\lambda^*}^-(w_{\delta'})-\hat{J}_{\lambda^*}^-\le \delta'.$
	
	Once $J_{\lambda_n}^-(w_{\delta'})\to J_{\lambda^*}^-(w_{\delta'})$ (see Corollary \ref{decrea}), we conclude that for sufficiently large $n$
	
	$$
	0\le J_{\lambda_n}^-(w_{\delta'})-J_{\lambda^*}^-(w_{\delta'})\le \delta'.
	$$
	
	It follows that for sufficiently large $n$, 
	
	$$
	\hat{J}_{\lambda_n}^-\le J_{\lambda_n}^-(w_{\delta'})\le \hat{J}_{\lambda^*}^-+2\delta'\le J-\delta+2\delta',
	$$ 
	and hence $J\le J-\delta+\delta'<J$, a contradiction, therefore  $J=\hat{J}_{\lambda^*}^-$.
	
	The proof is similar for $\hat{J}_{\lambda^*}^+$.
	
\end{proof}

Now we are able to show the existence of solutions to the minimization problems $\hat{J}_{\lambda^*}^\mp$.

\begin{prop}\label{lambdaextremal} There are function $w_{\lambda^*}\in \mathcal{N}_{\lambda^*}^-$ and  $u_{\lambda^*}\in \mathcal{N}_{\lambda^*}^+$ such that $\hat{J}_{\lambda^*}^-=J_{\lambda^*}^-(w_{\lambda^*})$ and $\hat{J}_{\lambda^*}^+=J_{\lambda^*}^+(u_{\lambda^*})$.
\end{prop}

\begin{proof} Take $\lambda_n\uparrow \lambda^*$ and $w_n\in \mathcal{N}_{\lambda_n}^-$ with $\hat{J}_{\lambda_n}^-=J_{\lambda_n}^-(w_n)$. Observe from the Proposition \ref{constramin} that 
	
	\begin{equation}\label{lambdaextremal1}
	-\Delta_p w_n-\lambda_n|w_n|^{q-2}w_n-f|w_n|^{\gamma-2}w_n=0,\ \forall\ n=1,2,\ldots.
	\end{equation}
	
	We claim that there exists postive constants $c,C$ such that $c\le \|w_n\|\le C$ for all $n=1,2,\ldots$. Indeed, from the Corollary \ref{estimatives} we only have to show existence of $C$, thus, suppose on the contrary that, up to a subsequence, $\|w_n\|\to \infty $ as $n\to \infty$. It follows from the Proposition \ref{contiex} and \eqref{coer} that

	\begin{align*}
	\hat{J}^-_{\lambda^*}=& \lim_{n\to \infty} J_{\lambda_n}^- \\ 
	=& \lim_{n\to \infty} J_{\lambda_n}^-(w_n) \\
	\ge&\lim_{n\to \infty}\left[ \left(\frac{1}{p}-\frac{1}{\gamma}\right)\|w_n\|^p-\left(\frac{1}{q}-\frac{1}{\gamma}\right)\lambda_n \int | w_n|^q\right] \\
	=&\infty,
	\end{align*}
	which is an absurd. Therefore, we can suppose that $c\le \|w_n\|\le C$ for all $n=1,2,\ldots$ and up to a subsequence $w_n \rightharpoonup w$ in $W_0^{1,p}(\Omega)$ and $w_n \to w$ in $L^q(\Omega),L^\gamma(\Omega)$. We claim that $w\neq 0$ and $F(w)>0$. In fact, if $w=0$ then from the Proposition \ref{inequal} we obtain that $\|w_n\|\to 0$ which is an absurd. 
	
	From \eqref{lambdaextremal1} and the $S^+$ property of the $p$-Laplacian (see \cite{drabekMilota}) we conclude that $w_n \to w$ in $W_0^{1,p}(\Omega)$ and 
	
	\begin{equation}\label{lambdaextremal2}
	-\Delta_p w-\lambda^*|w|^{q-2}w-f|w|^{\gamma-2}w=0.
	\end{equation}
	
	We claim that $w\in \mathcal{N}_{\lambda^*}^-$. If not then  $w\in \mathcal{N}_{\lambda^*}^0$. From the Proposition \ref{extre} we conclude that

	\begin{equation}\label{contra}
	-p\Delta_p w-\lambda^* q|w|^{q-2}w-\gamma f|w|^{\gamma-2}w=0.	
	\end{equation}	
	
	Let us prove that \eqref{contra} gives us an absurd. From \eqref{lambdaextremal1} and \eqref{contra} we obtain that 
	
	\begin{equation}\label{contra1}
	f(x)|w(x)|^{\gamma-q}=\frac{p-q}{\gamma-q}\lambda^*,\ a.e.\ x\in \{x\in \Omega:\ w(x)\neq 0\}.
	\end{equation}
	
	From the Corollary \ref{solureg}, we can assume that $w\in C(\overline{\Omega})$. Once $w\in W_0^{1,p}(\Omega)$, given $\varepsilon>0$, there exists $\delta>0$ such that if $\Omega_\delta=\{x\in \Omega:\ \operatorname{dist}(x,\partial\Omega)<\delta\}$ then $|w(x)|\le \varepsilon$, however, this contradicts \eqref{contra1} and the fact that $f\in L^\infty(\Omega)$. Therefore $w\in \mathcal{N}_{\lambda^*}^-$. It follows that 
	
	\begin{equation*}
	\hat{J}_{\lambda^*}^-=\lim \hat{J}_{\lambda_n}^-=\lim J_{\lambda_n}^-(w_n)=J^-_{\lambda^*}(w).
	\end{equation*}
	
	Now take $\lambda_n\uparrow \lambda^*$ and $u_n\in \mathcal{N}_{\lambda_n}^+$ with $\hat{J}_{\lambda_n}^+=J_{\lambda^*}^+(u_n)$. Observe from the Proposition \ref{constramin} that 
	
	\begin{equation}\label{lambdaextremal11}
	-\Delta_p u_n-\lambda_n|u_n|^{q-2}u_n-f|u_n|^{\gamma-2}u_n=0,\ \forall\ n=1,2,\ldots.
	\end{equation}
	
	We claim that there exists positive constants $c,C$ such that $c\le \|u_n\|\le C$ for all $n=1,2,\ldots$. Indeed, from the Corollary \ref{estimatives} we only have to show existence of $c$, thus, suppose on the contrary that, up to a subsequence, $\|u_n\|\to 0 $ as $n\to \infty$. It follows from the Proposition \ref{contiex} and \eqref{coer} that

	\begin{align*}
	\hat{J}^+_{\lambda^*}=& \lim_{n\to \infty} J_{\lambda_n}^+ \\ 
	=& \lim_{n\to \infty} J_{\lambda^*}^+(u_n) \\
	\ge&\lim_{n\to \infty}\left[ \left(\frac{1}{p}-\frac{1}{\gamma}\right)\|u_n\|^p-\left(\frac{1}{q}-\frac{1}{\gamma}\right)\lambda_n \int | u_n|^q\right] \\
	\ge &0,
	\end{align*}
	which is an absurd. Therefore, we can suppose that $c\le \|u_n\|\le C$ for all $n=1,2,\ldots$ and up to a subsequence $u_n \rightharpoonup u$ in $W_0^{1,p}(\Omega)$ and $u_n \to u$ in $L^q(\Omega),L^\gamma(\Omega)$. We claim that $u\neq 0$. In fact, if $u=0$ then from the Proposition \ref{inequal} we obtain that $\|u_n\|\to 0$ which is an absurd. 
	
	From \eqref{lambdaextremal11} and the $S^+$ property of the $p$-Laplacian we conclude that $u_n \to u$ in $W_0^{1,p}(\Omega)$ and 
	
	\begin{equation}\label{lambdaextremal22}
	-\Delta_p u-\lambda^*|u|^{q-2}u-f|u|^{\gamma-2}u=0.
	\end{equation}

	We claim that $u\in \mathcal{N}_{\lambda^*}^+$. If not then  $u\in \mathcal{N}_{\lambda^*}^0$. From the Proposition \ref{extre} we conclude that
	
	$$
	-p\Delta_p u-\lambda^* q|u|^{q-2}u-\gamma f|u|^{\gamma-2}u=0.	
	$$
	However this equation contradicts \eqref{lambdaextremal22} and consequently $u\in \mathcal{N}_{\lambda^*}^+$. It follows that 
	
	\begin{equation*}
	\hat{J}_{\lambda^*}^+=\lim \hat{J}_{\lambda_n}^+=\lim J_{\lambda_n}^+(u_n)=J^+_{\lambda^*}(u).
	\end{equation*}
	
	By taking $w_{\lambda^*}\equiv w$ and $u_{\lambda^*}\equiv u$, the proof is completed.
	
\end{proof}

Now we prove the Lemma \ref{positivesol1}.

\begin{proof}[Proof of the Lemma \ref{positivesol1}] From the Propositions \ref{minglob} and \ref{lambdaextremal}, for each $\lambda\in(0,\lambda^*]$, there exists $w_{\lambda}\in \mathcal{N}_\lambda^-$ and $u_\lambda \in \mathcal{N}_\lambda^+$ such that $J_\lambda^-(w_\lambda)=\hat{J}_\lambda^-$ and $J_\lambda^+(u_\lambda)=\hat{J}_\lambda^+$. 
	
	From the Proposition \ref{constramin} we have that both $w_{\lambda},u_{\lambda}$ are solutions of \eqref{pq} and $w_{\lambda},u_{\lambda}\in C^{1,\alpha}(\overline{\Omega})$ for some $\alpha\in (0,1)$. Moreover, once $\Phi_{\lambda}(u)= \Phi_{\lambda}(|u|)$ for all $u\in W_0^{1,p}(\Omega)$, it follows that $|w_{\lambda}|\in \mathcal{N}_\lambda^-$, $|u_\lambda| \in \mathcal{N}_\lambda^+$ and $J_\lambda^-(|w_\lambda|)=\hat{J}_\lambda^-$. $J_\lambda^+(|u_\lambda|)=\hat{J}_\lambda^+$, therefore, we can assume that $w_{\lambda},u_{\lambda}\ge 0$.
	
	From the Harnack inequality (see \cite{trud}) we obtain $w_{\lambda},u_{\lambda}>0$.

\end{proof}

%%%%%%%%%%%%%%%%%%%%%%%%%%%%%%%%%%%%%%%%%%%%%%%%%%%%%%%%%%%%%%%%%%%%%%%%%%%%%%%%%%%%%%%%%%%%%%%%%%%%%%%%%%%%%%%

\section{Existence of solutions for $\lambda>\lambda^*$}

In this section we show existence of solutions to the problem \eqref{pq} for $\lambda$ close to $\lambda^*$. In fact, we show that for $\lambda$ near $\lambda^*$, it is possible to minimize $\Phi_{\lambda}$ over submanifolds of the Nehari manifolds $\mathcal{N}_\lambda^-$ and $\mathcal{N}_\lambda^+$. 

\begin{lem}\label{positivesol2} There exists $\varepsilon>0$ such that for each $\lambda\in (\lambda^*,\lambda^*+\varepsilon)$, there exists $0<w_\lambda\in \mathcal{N}_\lambda^-$ and $0<u_\lambda\in \mathcal{N}_\lambda^+$ solutions of \eqref{pq}.
\end{lem}

The proof will be given at the end of this section. 

For $\lambda>0$, denote

\begin{equation*}\label{hlambda-}
H_\lambda^-(w)=	p\|w\|^p-\lambda q\|w\|_q^q-\gamma F(w),\ \forall\ w\in \mathcal{N}_\lambda^-\cup \mathcal{N}_\lambda^0,
\end{equation*}

and 
$$
H_\lambda^+(u)=	p\|u\|^p-\lambda q\|u\|_q^q-\gamma F(u),\ \forall\ u\in \mathcal{N}_\lambda^+\cup \mathcal{N}_\lambda^0.
$$ 

\begin{prop}\label{prope1} Let $0<c<C$. Assume that $\lambda_n\downarrow \lambda^*$.
	
	\begin{description}
		\item[(i)] suppose that $w_n\in \mathcal{N}_{\lambda^*}^-$ satisfies $c\le \|w_n\|\le C$ for all $n=1,2,\ldots$. If $H_{\lambda_n}^-(t_{\lambda_n}^-(w_n)w_n)\to 0$ then $\operatorname{dist}(w_n,\mathcal{N}_{\lambda^*}^0)\to 0$ as $n\to \infty$;
		\item[(ii)] suppose that $u_n\in \mathcal{N}_{\lambda^*}^+$ satisfies $c\le \|u_n\|\le C$ for all $n=1,2,\ldots$. If $H_{\lambda_n}^+(t_{\lambda_n}^+(u_n)u_n)\to 0$ then $\operatorname{dist}(u_n,\mathcal{N}_{\lambda^*}^0)\to 0$ as $n\to \infty$.
	\end{description}
\end{prop}

\begin{proof} $\mathbf{(i)}$ First observe from the Corollary \ref{estimatives} that there exists a positive constant $c$ such that $F(w_n)\ge c$ for all $n=1,2,\ldots$. We claim that the same holds for $\|w_n\|_q^q$. In fact, let us first prove that $t_{\lambda_n}^+(w_n)\to 1$. Observe that

	\begin{equation*}
	\left\{
	\begin{aligned}
	t_n^p\|w_n\|^p-\lambda_nt_n^q\|w_n\|_q^q-t_n^\gamma F(w_n)&= 0, \\
	pt_n^p\|w_n\|^p-\lambda_nqt_n^q\|w_n\|_q^q-\gamma t_n^\gamma F(w_n)& =o(1), \\
	s_n^p\|w_n\|^p-\lambda_ns_n^q\|w_n\|_q^q-s_n^\gamma F(w_n)&= 0,
	\end{aligned}
	\right. \ \forall\ n=1,2,\ldots,
	\end{equation*}	
	where $t_n=t_{\lambda_n}^-(w_n)$ and $s_n=t_{\lambda_n}^+(w_n)$. It follows that 
	
	$$
	\|w_n\|^p\left[p-q-(\gamma-q)\frac{1}{t_n^{\gamma-p}}\left(\frac{\left(\frac{s_n}{t_n}\right)^{p-q}-1}{\left(\frac{s_n}{t_n}\right)^{\gamma-q}-1}\right)\right]=o(1),\ n\to \infty.
	$$
	
	Since $\|w_n\|^p\ge c$ for $n=1,2,\ldots$, we conclude that $s_n,t_n\to 1$ as $n\to \infty$ and from the Corollary \ref{estimatives} we obtain that $\|w_n\|_q^q\ge c$ for all $n=1,2,\ldots$. Moreover, as $t_n\to 1$, we obtain 
	
	\begin{equation}\label{preope1}
	\left\{
	\begin{aligned}
	\|w_n\|^p-\lambda^*\|w_n\|_q^q- F(w_n)&= 0, \\
	p\|w_n\|^p-\lambda^*q\|w_n\|_q^q-\gamma  F(w_n)& =o(1), 	
	\end{aligned}
	\right. \ \forall\ n=1,2,\ldots,
	\end{equation}
	
	From \eqref{preope1} we produce the following identities 
	
	$$
	\frac{\gamma-p}{\gamma-q}\frac{\|w_n\|^p}{\|w_n\|_q^q}=\lambda^*+\frac{o(1)}{(\gamma-q)\|w_n\|_q^q},\ n\to \infty,
	$$
	and
	
	$$
	\frac{p-q}{\gamma-q}\frac{\|w_n\|^p}{F(w_n)}=1+\frac{o(1)}{(\gamma-q)F(w_n)},\ n\to \infty.
	$$
	
	From \eqref{rayleigh} we infer that 
	
	$$
	\lambda(w_n)=\left(\lambda^*+\frac{o(1)}{(\gamma-q)\|w_n\|_q^q}\right)\left(1+\frac{o(1)}{(\gamma-q)F(w_n)}\right)^{\frac{p-q}{\gamma-p}},\ n\to \infty.
	$$
	
	Therefore $\lambda(w_n)\to \lambda^*$ and $w_n$ is a bounded minimizing sequence for $\lambda^*$. Moreover, following the same argument of the item $\bf (ii)$ of the Proposition \ref{extre} we can  see that, up to a subsequence, $w_n\to w\in \mathcal{N}_{\lambda^*}^0$ and consequently $\operatorname{dist}(w_n,\mathcal{N}_{\lambda^*}^0)\to 0$ as $n\to \infty$.
	
	$\mathbf{(ii)}$ Indeed, first observe from the Corollary \ref{estimatives} that there exists a positive constant $c$ such that $\|u_n\|_q\ge c$ for all $n=1,2,\ldots$. We claim that the same holds for $F(u_n)$. In fact, let us first prove that $t_{\lambda_n}^-(u_n)\to 1$. Observe that

	\begin{equation*}
	\left\{
	\begin{aligned}
	t_n^p\|u_n\|^p-\lambda_nt_n^q\|u_n\|_q^q-t_n^\gamma F(u_n)&= 0, \\
	pt_n^p\|u_n\|^p-\lambda_nqt_n^q\|u_n\|_q^q-\gamma t_n^\gamma F(u_n)& =o(1), \\
	s_n^p\|u_n\|^p-\lambda_ns_n^q\|u_n\|_q^q-s_n^\gamma F(u_n)&= 0,
	\end{aligned}
	\right. \ \forall\ n=1,2,\ldots,
	\end{equation*}	
	where $t_n=t_{\lambda_n}^+(u_n)$ and $s_n=t_{\lambda_n}^-(u_n)$ . It follows that 
	
	$$
	\|u_n\|^p\left[p-q-(\gamma-q)\frac{1}{t_n^{\gamma-p}}\left(\frac{\left(\frac{s_n}{t_n}\right)^{p-q}-1}{\left(\frac{s_n}{t_n}\right)^{\gamma-q}-1}\right)\right]=o(1),\ n\to \infty.
	$$
	
	Once $\|u_n\|^p\ge c$ for $n=1,2,\ldots$, we conclude that $s_n,t_n\to 1$ as $n\to \infty$ and from the Corollary \ref{estimatives} we obtain that $F(u_n)\ge c$ for all $n=1,2,\ldots$. Therefore
	
	\begin{equation}\label{preope2}
	\left\{
	\begin{aligned}
	\|u_n\|^p-\lambda^*\|u_n\|_q^q- F(u_n)&= 0, \\
	p\|u_n\|^p-\lambda^*q\|u_n\|_q^q-\gamma  F(u_n)& =o(1), 	
	\end{aligned}
	\right. \ \forall\ n=1,2,\ldots,
	\end{equation}
	
	From \eqref{preope2} we produce the following identities 
	
	$$
	\frac{\gamma-p}{\gamma-q}\frac{\|u_n\|^p}{\|u_n\|_q^q}=\lambda^*+\frac{o(1)}{(\gamma-q)\|u_n\|_q^q},\ n\to \infty,
	$$
	and
	
	$$
	\frac{p-q}{\gamma-q}\frac{\|u_n\|^p}{F(u_n)}=1+\frac{o(1)}{(\gamma-q)F(u_n)},\ n\to \infty.
	$$
	
	From \eqref{rayleigh} we obtain that 
	
	$$
	\lambda(u_n)=\left(\lambda^*+\frac{o(1)}{(\gamma-q)\|u_n\|_q^q}\right)\left(1+\frac{o(1)}{(\gamma-q)F(u_n)}\right)^{\frac{p-q}{\gamma-p}},\ n\to \infty.
	$$
	
	Therefore $\lambda(u_n)\to \lambda^*$, which implies that $u_n$ is a bounded minimizing sequence for $\lambda^*$. Moreover, following the same argument of the item $\bf (ii)$ of the Proposition \ref{extre} we can  see that, up to a subsequence, $u_n\to u\in \mathcal{N}_{\lambda^*}^0$ and consequently $\operatorname{dist}(u_n,\mathcal{N}_{\lambda^*}^0)\to 0$ as $n\to \infty$.
	
\end{proof}

Consider the sets

\begin{equation*}\label{vacavaca}
\mathcal{N}_{\lambda^*,d,C}^-\equiv \{w\in \mathcal{N}_{\lambda^*}^-:\ 	\operatorname{dist}(\{w,|w|\},\mathcal{N}_{\lambda^*}^0)>d,\ \|w\|\le C\},
\end{equation*}

where $d>0$ and $C>0$. Similar, define

$$
\mathcal{N}_{\lambda^*,d,c}^+\equiv \{u\in \mathcal{N}_{\lambda^*}^+:\ 	\operatorname{dist}(\{u,|u|\},\mathcal{N}_{\lambda^*}^0)>d,\ \|u\|\ge c\},
$$
where $d>0$ and and $c>0$. 

\begin{cor}\label{unif3} There holds
	
	\begin{description}
		\item[(i)] take $d>0$ and $C>0$. There exists $\varepsilon>0$ such that if $w\in \mathcal{N}_{\lambda^*,d,C}^-$ then $w\in \hat{\mathcal{N}}_\lambda$ for all $\lambda\in (\lambda^*,\lambda^*+\varepsilon)$. Moreover, there exists $\delta<0$ such that  $H_{\lambda}^-(t_\lambda^-(w)w)<\delta$ for all $w\in \mathcal{N}_{\lambda^*,d,C}^-$;
		\item[(ii)] take $d>0$ and $c>0$. There exists $\varepsilon>0$ such that if $u\in \mathcal{N}_{\lambda^*,d,c}^+$ then $u\in \hat{\mathcal{N}}_\lambda\cup \hat{\mathcal{N}}_\lambda^+$ for all $\lambda\in (\lambda^*,\lambda^*+\varepsilon)$. Moreover, there exists $\delta>0$ such that  $H_{\lambda}^+(t_\lambda^+(w)w)>\delta$ for all $w\in \mathcal{N}_{\lambda^*,d,c}^+$.
	\end{description}
\end{cor}

\begin{proof} Immediately from the Proposition \ref{preope1}.

\end{proof}

The Corollary \ref{unif3} shows that for $\lambda$ close to $\lambda^*$, the Nehari submanifolds $\mathcal{N}_{\lambda^*,d,C}^-$ and $\mathcal{N}_{\lambda^*,d,c}^+$ projects over the Nehari manifolds $\mathcal{N}_{\lambda}^-$ and $\mathcal{N}_{\lambda}^+$ respectively. 

For each $\lambda\in (0,\infty)$, denote 

\begin{equation*}\label{boiboi}
\mathcal{S}_\lambda^{-}=\{w\in \mathcal{N}_\lambda^-:\ J^{-}_\lambda(w)=\hat{J}_\lambda^{-}\}
\end{equation*}
and 
$$
\mathcal{S}_\lambda^{+}=\{u\in \mathcal{N}_\lambda^+:\ J^{+}_\lambda(u)=\hat{J}_\lambda^{+}\}.
$$

From the previous section we know that $\mathcal{S}_\lambda^{\mp}\neq \emptyset$ for all $\lambda\in (0,\lambda^*]$.

\begin{prop}\label{nearlam} There holds 
	\begin{description}
		\item[(i)] $$
		\operatorname{dist}(\mathcal{S}_{\lambda^*}^-,\mathcal{N}_{\lambda^*}^0)>0;
		$$
		\item[(ii)] $$
		\operatorname{dist}(\mathcal{S}_{\lambda^*}^+,\mathcal{N}_{\lambda^*}^0)>0.
		$$
	\end{description}
	
\end{prop}

\begin{proof} $\mathbf{(i)}$ Suppose on the contrary that $\operatorname{dist}(\mathcal{S}_{\lambda^*}^-,\mathcal{N}_{\lambda^*}^0)=0$. Therefore, we can find a sequence $w_n\in \mathcal{S}_{\lambda^*}^-$ and a corresponding sequence $v_n \in \mathcal{N}_{\lambda^*}^0$ such that $\|w_n-v_n\|\to 0$ as $n\to \infty$ and
	
	\begin{equation}\label{nearlam1}
	-\Delta_p w_n-\lambda^*|w_n|^{q-2}w_n-f|w_n|^{\gamma-2}w_n=0,\ \forall\ n=1,2,\ldots.
	\end{equation}
	
	From the Proposition \ref{compact} we can assume without loss of generality that $v_n \to v\in \mathcal{N}_{\lambda^*}^0$ and hence $w_n\to v$. Passing the limit in \eqref{nearlam1} we obtain that 
	
	$$
	-\Delta_p v-\lambda^*|v|^{q-2}v-f|v|^{\gamma-2}v=0,
	$$
	however, once $v\in \mathcal{N}_{\lambda^*}^0$, we know from the Proposition \ref{extre} that 
	
	$$
	-p\Delta_p v-\lambda^* q|v|^{q-2}v-\gamma f|v|^{\gamma-2}v=0,		
	$$
	which is a contradiction.
	
	The proof is similar for $\mathbf{(ii)}$.
	
\end{proof}

Define $d^-_{\lambda^*}\equiv 	\operatorname{dist}(\mathcal{S}_{\lambda^*}^-,\mathcal{N}_{\lambda^*}^0)$ and $d^+_{\lambda^*}\equiv 	\operatorname{dist}(\mathcal{S}_{\lambda^*}^+,\mathcal{N}_{\lambda^*}^0)$.

Choose $C_{\lambda^*}>0$ such that $\|w\|\le C_{\lambda^*}$ for all $w\in \mathcal{S}_{\lambda^*}^-$. Take $d^-\in (0,d^-_{\lambda^*})$, $C>C_{\lambda^*}$ and $\varepsilon>0$ as in the Corollary \ref{unif3}. Define for $\lambda\in (\lambda^*,\lambda^*+\varepsilon)$

\begin{equation*}\label{jmenoslambdadeceparalembraotrem}
\hat{J}^-_{\lambda,d^-,C}=\inf\{J^-_{\lambda}(w):\ w\in \mathcal{N}_{\lambda^*,d^-,C}^-\}.
\end{equation*}

Similar choose $c_{\lambda^*}>0$ such that $c_{\lambda^*}\le \|u\|$ for all $u\in \mathcal{S}_{\lambda^*}^-$. Take $d^+\in (0,d^+_{\lambda^*})$, $c<c_{\lambda^*}$ and $\varepsilon>0$ as in the Corollary \ref{unif3}. Define for $\lambda\in (\lambda^*,\lambda^*+\varepsilon)$ 

$$
\hat{J}^+_{\lambda,d^+,c}=\inf \{J^+_{\lambda}(u):\ u\in \mathcal{N}_{\lambda^*,d^+,c}^+\}.
$$

Observe from the Proposition \ref{nearlam} that for each $d^-,d^+,c,C$ satisfying the above conditions we have that $\mathcal{S}_{\lambda^*}^-\subset \mathcal{N}_{\lambda^*,d^-,C}^-$ and  $\mathcal{S}_{\lambda^*}^+\subset \mathcal{N}_{\lambda^*,d^+,c}^+$.

\begin{prop}\label{cvac} There holds 
	
	\begin{description}
		\item[(i)] $$
		\lim_{\lambda\downarrow \lambda^*}\hat{J}^-_{\lambda,d^-,C}=\hat{J}_{\lambda^*}^-;
		$$
		\item[(ii)]  $$
		\lim_{\lambda\downarrow \lambda^*}\hat{J}^+_{\lambda,d^+,c}=\hat{J}_{\lambda^*}^+.
		$$
	\end{description}

\end{prop}

\begin{proof} $\mathbf{(i)}$ From the Proposition \ref{decrea1}, we have $\hat{J}^-_{\lambda,d^-,C}\le J^- _{\lambda}(w)< J^-_{\lambda'}(w)$ for all $w\in \mathcal{N}_{\lambda^*,d^-,C}^-$ and $\lambda^*<\lambda'<\lambda<\lambda^*+\varepsilon$ and hence $\hat{J}^-_{\lambda,d^-,C}\le \hat{J}^-_{\lambda',d^-,C}$. Moreover, if $w_{\lambda^*}\in\mathcal{S}_{\lambda^*}^-$ then for all $\lambda\in (\lambda^*,\lambda^*+\varepsilon)$ we have that $\hat{J}^-_{\lambda,d^-,C}\le J^- _{\lambda}(w_{\lambda^*})< J^-_{\lambda^*}(w_{\lambda^*})=\hat{J}_{\lambda^*}^-$.

	Take $\lambda_n\downarrow \lambda^*$ and suppose ad absurdum that $\hat{J}_{\lambda_n,d^-,C}^-$ does not converge to $\hat{J}_{\lambda^*}^-$. We can assume without loss of generality that $\hat{J}_{\lambda_n,d^-,C}^-\to J< \hat{J}_{\lambda^*}^-$ as $n\to \infty$.

	For each $n=1,2,\ldots,$ choose $w_n\in \mathcal{N}_{\lambda^*,d^-,C}^-$ such that $J_{\lambda_n}^-(w_n)-\hat{J}_{\lambda_n,d^-,C}^-\le 1/2^n$.

	Once $\|w_n\|$ is bounded, we can assume that up to a subsequence $w_n \rightharpoonup w$ in $W_0^{1,p}(\Omega)$ and $w_n \to w$ in $L^q(\Omega),L^\gamma(\Omega)$. Note that $w\neq 0$. In fact, if $w=0$ then from the Proposition \ref{inequal} we obtain that $\|w_n\|\to 0$ which is an absurd.  We claim that $w_n\to w$ in $W_0^{1,p}(\Omega)$. In fact, on the contrary, we would have that $\|w\|<\liminf \|w_n\|$ and thus 
	
	$$
	\liminf_{n\to \infty} D_u\Phi_{\lambda_n}\left(t_{\lambda^*}(w)w_n\right)>  D_u\Phi_{\lambda^*}\left(t_{\lambda^*}(w)w\right)=0,
	$$
	for $t_{\lambda^*}(u)$ defined as in \eqref{t-lambda*}, which implies that for sufficiently large $n$, $  D_u\Phi_{\lambda_n}\left(t_{\lambda^*}(w)w_n\right)>0$. Therefore, for sufficiently large $n$ we have that $t^+_{\lambda_n}(w_n)<t_{\lambda^*}\left(w\right)<t_{\lambda_n}^-(w_n)$ and hence

	\begin{align*}
	\Phi_{\lambda^*}\left(t_{\lambda^*}(w)w\right)       &<  \liminf_{n\to \infty} \Phi_{\lambda_n}\left(t_{\lambda^*}\left(w\right)w_n\right) \\
	&< \liminf_{n\to \infty} \hat{J}_{\lambda_n,d^-,C}^-=J,                   
	\end{align*}
	which is an absurd, because from the Proposition \ref{homeo} and the Corollary \ref{emptyinterior} we have that $\Phi_{\lambda^*}\left(t_{\lambda^*}(w)w\right)\ge \hat{J}_{\lambda^*}^-$. It follows that $w_n\to w$ in $W_0^{1,p}(\Omega)$ and consequently, from the Proposition \ref{uniconv} we conclude that $|J^-_{\lambda_n}(w_n)-J^-_{\lambda^*}(w_n)|\to 0$ as $n\to \infty$, which is a contradiction. 
	
	$\mathbf{(ii)}$ Similar to $\mathbf{(i)}$.
	
\end{proof}

\begin{prop} \label{loca1} Take $d^-\in (0,d^-_{\lambda^*})$ and $C>C_{\lambda^*}$. There exists $\varepsilon^->0$ such that for all $\lambda\in (\lambda^*,\lambda^*+\varepsilon^-)$, the problem $\hat{J}^-_{\lambda,d^-,C}$ has a minimizer $w_\lambda\in \mathcal{N}_{\lambda,d^-,C}^-$.
\end{prop}

\begin{proof} For each $\lambda>0$, let $w_n(\lambda)\in \mathcal{N}_{\lambda^*,d^-,C}^-$ be a minimizing sequence for $\hat{J}^-_{\lambda,d^-,C}$. From the Corollary \ref{unif3} we can assume that $t_{\lambda}^-(w_n(\lambda))\to t(\lambda)\in (0,1)$ and $w_n(\lambda) \rightharpoonup w(\lambda)\neq 0$ in $W_0^{1,p}(\Omega)$. Let us prove that there exists $\varepsilon^->0$ such that $w(\lambda)\in \hat{\mathcal{N}}_\lambda$ for all $\lambda\in (\lambda^*,\lambda^*+\varepsilon^-)$. Suppose on the contrary that there exists a sequence $\lambda_m\downarrow \lambda^*$ such that $w(\lambda_m)\notin \hat{\mathcal{N}}_{\lambda_m}$ for all $m=1,2,\ldots$
	
	Denote $w_{n,m}\equiv t_{\lambda_m}^-(w_n(\lambda_m)) w_n(\lambda_m)$. If necessary, by relabeling the sequence $w_{n,m}$, we can assume that 
	
	\begin{equation}\label{solextre}
	|\hat{J}^-_{\lambda_m,d^-,C}-\hat{J}_{\lambda_m}^-(w_{n,m})|\le \frac{1}{2^m},\ n,m=1,2,\ldots.
	\end{equation}
	
	From \eqref{solextre} and the Proposition \ref{cvac} we conclude that

	\begin{equation}\label{solextre1}
	|\hat{J}_{\lambda^*}^--J_{\lambda_m}^-(w_{n,m})|\le |\hat{J}_{\lambda^*}^--\hat{J}^-_{\lambda_m,d^-,C}|+|\hat{J}^-_{\lambda_m,d^-,C}-J_{\lambda_m}^-(w_{n,m})|\to 0,\ n,m\to \infty.
	\end{equation}

	From the Corollary \ref{estimatives} we can assume that $0<c\le t_{\lambda_m}^-(w_{n,m})< 1$ for all $n,m=1,2,\ldots$, therefore we can suppose that $w_{n,m}\rightharpoonup w$ in $W_0^{1,p}(\Omega)\setminus \{0\}$ as $n,m\to \infty$ and $w_{n,n}\to w$ in $L^p(\Omega),L^\gamma(\Omega)$. We claim that $w_{n,m}\to w$ in $W_0^{1,p}(\Omega)\setminus \{0\}$ as $n,m\to \infty$. Indeed, suppose not then, $\|w\|<\liminf_{n,m}\|w_{n,m}\|$ and 
	
	$$
	\liminf_{n,m\to \infty}D_u\Phi_{\lambda_m}(t_{\lambda^*}(w)w_{n,m})>\Phi_{\lambda^*}(t_{\lambda^*}(w)w)=0,
	$$
	for $t_{\lambda^*}(u)$ defined as in \eqref{t-lambda*}. Hence, for $n,m$ sufficiently large, we can suppose that $D_u\Phi_{\lambda_m}(t_{\lambda^*}(w)w_{n,m})>0$. It follows that for $n,m$ sufficiently large, $t_{\lambda_m}^+(w_{n,m})<t_{\lambda^*}(w)<t_{\lambda_m}^-(w_{n,m})$. Therefore, from \eqref{solextre1}
	
	\begin{align*}
	\Phi_{\lambda^*}(t_{\lambda^*}(w)w)<&\liminf_{n,m\to\infty} \Phi_{\lambda_m}(t_{\lambda^*}(w)w_{n,m}) \\
	<& \liminf_{n,m\to\infty} J^-_{\lambda_m}(w_{n,m}) \\
	=& \hat{J}^-_{\lambda^*}
	\end{align*}
	which is an absurd and hence $w_{n,m}\to w$ in $W_0^{1,p}(\Omega)\setminus \{0\}$ as $n,m\to \infty$. Hence, if $w_m\equiv w(\lambda_m)$ we obtain that
	
	$$
	\|w_m-w\|\le \liminf _{n\to \infty}\|w_{n,m}-w\|,\ \forall\ m=1,2,\ldots,
	$$
	which implies that for sufficiently large $m$, the sequence $w_m$ belongs to $\mathcal{N}^-_{\lambda^*,d^-,C}$ and consequently $w_m\in \hat{\mathcal{N}}_{\lambda_m}$ for sufficiently large $m$, which is a contradiction. Therefore, there exists $\varepsilon^->0$ such that $w(\lambda)\in \hat{\mathcal{N}}_\lambda$ for all $\lambda\in (\lambda^*,\lambda^*+\varepsilon^-)$. Arguing as in the Proposition \ref{lambdaextremal}, we conclude that for all $\lambda\in (\lambda^*,\lambda^*+\varepsilon^-)$, we have  $t_{\lambda}^-(w_n(\lambda))w_n(\lambda)\to t(\lambda) w(\lambda)$ in  $W_0^{1,p}(\Omega)$, $w(\lambda)\in \mathcal{N}_{\lambda^*,d^-,C}^-$ and 
	
	$$
	J_{\lambda,d^-,C}^-=J_{\lambda}^-(w(\lambda)).
	$$
	
	By denoting $w_\lambda\equiv w(\lambda)$, the proof is complete.

\end{proof}

\begin{prop} \label{loca2} Take $d^+\in (0,d^+_{\lambda^*})$ and $c<c_{\lambda^*}$. There exists $\varepsilon^+>0$ such that for all $\lambda\in (\lambda^*,\lambda^*+\varepsilon^+)$, the problem $\hat{J}^+_{\lambda,d^+,c}$ has a minimizer $u_\lambda\in \mathcal{N}_{\lambda,d^+,c}^+$.
\end{prop}

\begin{proof} For each $\lambda>0$, let $u_n(\lambda)\in \mathcal{N}_{\lambda^*,d^+,c}^+$ be a minimizing sequence for $\hat{J}^+_{\lambda,d^+,c}$. From the Corollary \ref{unif3} we can assume that $t_{\lambda}^+(u_n(\lambda))\to t(\lambda)\in (1,\infty)$ and $u_n(\lambda) \rightharpoonup u(\lambda)\neq 0$ in $W_0^{1,p}(\Omega)$. Let us prove that there exists $\varepsilon^+>0$ such that $u(\lambda)\in \hat{\mathcal{N}}_\lambda\cup \hat{\mathcal{N}}^+_\lambda$ for all $\lambda\in (\lambda^*,\lambda^*+\varepsilon^+)$. Suppose on the contrary that there exists a sequence $\lambda_m\downarrow \lambda^*$ such that $u(\lambda_m)\notin \hat{\mathcal{N}}_{\lambda_m}\cup \hat{\mathcal{N}}^+_{\lambda_m}$ for all $m=1,2,\ldots$
	
	Denote $u_{n,m}\equiv t_{\lambda_m}^-(u_n(\lambda_m)) u_n(\lambda_m)$. If necessary, by relabeling the sequence $u_{n,m}$, we can assume that that 
	
	\begin{equation}\label{solextrea}
	|\hat{J}^+_{\lambda_m,d^+,c}-\hat{J}_{\lambda_m}^+(u_{n,m})|\le \frac{1}{2^m},\ n,m=1,2,\ldots.
	\end{equation}
	
	From \eqref{solextrea} and the Proposition \ref{cvac} we conclude that

	\begin{equation}\label{solextre1a}
	|\hat{J}_{\lambda^*}^+-J_{\lambda_m}^+(u_{n,m})|\le |\hat{J}_{\lambda^*}^+-\hat{J}^+_{\lambda_m,d^+,c}|+|\hat{J}^+_{\lambda_m,d^+,c}-J_{\lambda_m}^+(u_{n,m})|\to 0,\ n,m\to \infty.
	\end{equation}

	From the Corollary \ref{estimatives} we can assume that $1< t_{\lambda_m}^+(u_{n,m})\le C$ for all $n,m=1,2,\ldots$, therefore we can suppose without loss of generality that $u_{n,m}\rightharpoonup u$ in $W_0^{1,p}(\Omega)\setminus \{0\}$ as $n,m\to \infty$ and $u_{n,n}\to u$ in $L^p(\Omega),L^\gamma(\Omega)$. We claim that $u_{n,m}\to u$ in $W_0^{1,p}(\Omega)\setminus \{0\}$ as $n,m\to \infty$. Indeed, suppose not. Then $\|u\|<\liminf_{n,m}\|u_{n,m}\|$ and 
	$$
	\liminf_{n,m\to \infty}	D_u\Phi_{\lambda_m}(s_{\lambda^*}(u)u_{n,m})>\Phi_{\lambda^*}(s_{\lambda^*}(u)u)=0,
	$$
	for $s_{\lambda^*}(u)$ defined as in  \eqref{s-lambda*}. Hence, for $n,m$ sufficiently large, we can assume that $D_u\Phi_{\lambda_m}(s_{\lambda^*}(u)u_{n,m})>0$. It follows that for $n,m$ sufficiently large, $t_{\lambda_m}^+(u_{n,m})<s_{\lambda^*}(u)$. Therefore, from \eqref{solextre1a}
	\begin{align*}
	\Phi_{\lambda^*}(s_{\lambda^*}(u)u)<& \liminf_{n,m\to\infty} J^+_{\lambda_m}(u_{n,m}) \\
	=& \hat{J}^+_{\lambda^*}
	\end{align*}
	which is an absurd and hence $u_{n,m}\to u$ in $W_0^{1,p}(\Omega)\setminus \{0\}$ as $n,m\to \infty$. Therefore, if $u_m\equiv u(\lambda_m)$ we obtain that
	
	$$
	\|u_m-u\|\le \liminf _{n\to \infty}\|u_{n,m}-u\|,\ \forall\ m=1,2,\ldots,
	$$
	which implies that for sufficiently large $m$, the sequence $u_m$ belongs to $\mathcal{N}^+_{\lambda^*,d^+,c}$ and consequently $u_m\in \hat{\mathcal{N}}_{\lambda_m}\cup\hat{\mathcal{N}}^+_{\lambda_m} $ for sufficiently large $m$, which is a contradiction. Therefore, there exists $\varepsilon^+>0$ such that $u(\lambda)\in \hat{\mathcal{N}}_\lambda$ for all $\lambda\in (\lambda^*,\lambda^*+\varepsilon^+)$. Arguing as in the Proposition \ref{lambdaextremal}, we conclude that for all $\lambda\in (\lambda^*,\lambda^*+\varepsilon^+)$, we have  $t_{\lambda}^+(u_n(\lambda))u_n(\lambda)\to t(\lambda) u(\lambda)$ in  $W_0^{1,p}(\Omega)$, $u(\lambda)\in \mathcal{N}_{\lambda^*,d,c}^+$ and 
	
	$$
	J_{\lambda,d^+,c}^+=J_{\lambda}^+(u(\lambda)).
	$$
	
	By denoting $u_\lambda\equiv u(\lambda)$, the proof is complete.
\end{proof}

Now we prove the Lemma \ref{positivesol2}.

\begin{proof}[Proof of the Lemma \ref{positivesol2}]  Choose $d^-\in (0,d^-_{\lambda^*})$, $d^+\in (0,d^+_{\lambda^*})$, $C>C_{\lambda^*}$ and $c<c_{\lambda^*}$. From the Propositions \ref{loca1} and \ref{loca2}, for each $\lambda\in(\lambda^*,\lambda^*+\varepsilon)$, where $\varepsilon=\min\{\varepsilon^-,\varepsilon^+\}$, there exists $\overline{w}_{\lambda}\in \mathcal{N}_{\lambda^*,d^-,C}^-$ and $\overline{u}_\lambda \in \mathcal{N}_{\lambda^*,d^+,c}^+$ such that $J_\lambda^-(\overline{w}_\lambda)=J_{\lambda,d^-,C}^-$ and $J_\lambda^+(\overline{u}_\lambda)=J_{\lambda,d^+,c}^+$. 
	
	From the Corollary \ref{solureg} we have that both $w_\lambda\equiv t_\lambda^-(\overline{w}_\lambda)\overline{w}_{\lambda},u_\lambda\equiv t_\lambda^+(\overline{u}_\lambda)\overline{u}_{\lambda}$ are solutions of \eqref{pq} and $w_{\lambda},u_{\lambda}\in C^{1,\alpha}(\overline{\Omega})$ for some $\alpha\in (0,1)$. Moreover, once $\Phi_{\lambda}(u)= \Phi_{\lambda}(|u|)$ for all $u\in W_0^{1,p}(\Omega)$, it follows that $|\overline{w}_{\lambda}|\in \mathcal{N}_{\lambda,d^-,C}^-$, $|\overline{u}_\lambda| \in \mathcal{N}_{\lambda,d^+,c}^+$ and $J_\lambda^-(|w_\lambda|)=\hat{J}_{\lambda,d^-,c}^-$, $J_{\lambda,d^+,c}^+(|u_\lambda|)=\hat{J}_\lambda^+$, therefore, we can assume that $w_{\lambda},u_{\lambda}\ge 0$.	From the Harnack inequality (see \cite{trud}) we obtain $w_{\lambda},u_{\lambda}>0$.

\end{proof}	

%%%%%%%%%%%%%%%%%%%%%%%%%%%%%%%%%%%%%%%%%%%%%%%%%%%%%%%%%%%%%%%%%%%%%%%%%%%%%%%%%%%%%%%%%%%%%%%%%%%%%

\section{Behavior of $u_{\lambda}$ near $\lambda=0$}

From the Lemma \ref{minglob} we have that $\mathcal{S}_\lambda^{+}\neq \emptyset $.

In this section we study the behavior of $\lambda^{-1/(p-q)}u$ near $\lambda=0$, where $u\in \mathcal{S}_\lambda^+$. Let $z\in W_0^{1,p}(\Omega)$ denote the unique positive solution of (see D\'iaz-S\'aa \cite{diaz})

\begin{equation}\label{pq1}\tag{$p,q$}
\left\{
\begin{aligned}
-\Delta_p u &=  |u|^{q-2}u
&\mbox{in}\ \ \Omega, \\
&u\in W_0^{1,p}(\Omega). \nonumber
\end{aligned}
\right.
\end{equation}

\begin{lem}\label{near0} Given $\varepsilon>0$, there exists $\delta>0$ such that if $0<\lambda<\delta$ then 
	
	$$
	\|\lambda^{-1/(p-q)}u-z\|\le \varepsilon,\ \forall\ u\in \mathcal{S}_\lambda^+.
	$$
\end{lem}

The proof will be given at the end of the section. Let $\mathcal{N}_0$ be the Nehari manifold associated with \eqref{pq1} then, one can easily see that 

$$
\mathcal{N}_0=\{\|v\|_q^{q/(p-q)}v:\ v\in S\}.
$$

\begin{prop}\label{near1} There holds
	
	$$
	\lim_{\lambda\to 0}\frac{t_\lambda^+(v)}{\lambda^{1/(p-q)}}=\|v\|_q^{q/(p-q)},
	$$
	uniformly in $v\in S$.

\end{prop}

\begin{proof} Indeed, once $t_\lambda^+(v)v\in \mathcal{N}_\lambda$, we have that 
	
	$$
	\frac{t_\lambda^+(v)^{p-q}}{\lambda}-\|v\|_q^q=\frac{t_\lambda^+(v)^{p-q}}{\lambda}t_\lambda^+(v)^{\gamma-p}F(v).
	$$
	
	From the Propostion \ref{estimatives} item $\mathbf{(ii)}$, there is some positive constant $C$ such that $t_\lambda^+(v)\le C\lambda^{1/(p-q)}$ and $\frac{t_\lambda^+(v)^{p-q}}{\lambda}\le C$ for $\lambda>0$. Therefore 
	
	$$
	\lim_{\lambda\downarrow 0}\left|\frac{t_\lambda^+(v)^{p-q}}{\lambda}-\|v\|_q^q \right|\le \lim_{\lambda\downarrow 0} C\lambda^{\frac{\gamma-p}{p-q}},
	$$
	which implies that 
	
	$$
	\lim_{\lambda\downarrow 0}\frac{t_\lambda^+(v)}{\lambda^{1/(p-q)}}=\lim_{\lambda\downarrow 0}\left(\frac{t_\lambda^+(v)^{p-q}}{\lambda}\right)^{1/(p-q)}=\|v\|_q^{q/(p-q)},
	$$
	uniformly in $v\in S$.	
\end{proof}

From the Proposition \eqref{near1} we obtain that $\mathcal{N}_\lambda^+/\lambda^{1/(p-q)}\to \mathcal{N}_0$ as $\lambda \downarrow 0$, to wit 

\begin{cor}\label{near2} 
	
	$$
	\lim_{\lambda\downarrow 0}\frac{t_\lambda^+(v)v}{\lambda^{1/(p-q)}}\to \|v\|_q^{q/(p-q)}v,
	$$
	uniformly in $v\in S$.
\end{cor}

Moreover, if $\Phi_0$ is the energy functional associated to \eqref{pq1} then, $\frac{J_\lambda^+}{\lambda^{p/(p-q)}}$ converge to $\Phi_0$ uniformly in $v\in S$, that is

\begin{cor}\label{near3}
	
	$$
	\lim_{\lambda\downarrow 0}\	\frac{J_\lambda^+(v)}{\lambda^{p/(p-q)}}= \Phi_0(\|v\|_q^{q/(p-q)}v),
	$$
	uniformly in $v\in S$.
\end{cor}

\begin{proof} In fact, we have that 
	
	$$
	\frac{J_\lambda^+(v)}{\lambda^{p/(p-q)}}=\frac{1}{p}\left(\frac{t_\lambda^+(v)}{\lambda^{1/(p-q)}}\right)^p-\frac{1}{q}\left(\frac{t_\lambda^+(v)}{\lambda^{1/(p-q)}}\right)^q\|v\|_q^q-\frac{\lambda^{\frac{\gamma-p}{p-q}}}{\gamma}\left(\frac{t_\lambda^+(v)}{\lambda^{1/(p-q)}}\right)^\gamma F(v).
	$$
	
	From the Corollary \eqref{near3} we conclude that 
	
	$$
	\lim_{\lambda\downarrow 0}\	\frac{J_\lambda^+(v)}{\lambda^{p/(p-q)}}=\frac{1}{p}\int |\|v\|_q^{q/(p-q)}v|^p-\frac{1}{q}\int |\|v\|_q^{q/(p-q)}v|^q=\Phi_0(\|v\|_q^{q/(p-q)}v),
	$$
	uniformly in $v\in S$.
	
\end{proof}

Denote 

$$
\hat{\Phi}_0=\inf \{\Phi_0(\|v\|_q^{q/(p-q)}v):\ v\in S\}.
$$

Let $\hat{z}=z/\|z\|$ and note that $\hat{\Phi}_0<0$ and $\Phi_0(\|\hat{z}\|_q^{q/(p-q)}\hat{z})=\hat{\Phi}_0$. Now we are ready to prove the Lemma \ref{near0}

\begin{proof}[Proof of Lemma \ref{near0}] Observe from the Corollary \ref{near3} that 
	
	\begin{equation}\label{infconv}
	\lim_{\lambda\downarrow 0}\frac{\hat{J}_\lambda^+}{\lambda^{p/(p-q)}} = \hat{\Phi}_0<0. 
	\end{equation}
	
	Let us prove that, given $\varepsilon>0$, there exists $\delta>0$ such that if $0<\lambda<\delta$ then 
	
	$$
	\|\lambda^{-1/(p-q)}t_\lambda^+(v)v-\|\hat{z}\|_q^{q/(p-q)}\hat{z}\|\le \varepsilon,\ \forall\ v\in \mathcal{S}_\lambda^+.
	$$
	
	Indeed, suppose not. Then, we can find a sequence $\lambda_n\downarrow 0$ and a corresponding sequence $v_n\in \mathcal{S}_{\lambda_n}^+$ such that

	\begin{equation}\label{absrd}
	\|\lambda_n^{-1/(p-q)}t_{\lambda_n}^+(v_n)v_n-\|\hat{z}\|_q^{q/(p-q)}\hat{z}\|> \varepsilon.
	\end{equation}

	From the Proposition \ref{estimatives} item $\mathbf{(ii)}$ we have that $\|\lambda_n^{-1/(p-q)}t_{\lambda_n}^+(v_n)v_n\|$ for $n=1,2\dots$ is bounded. Therefore we can assume that $\lambda_n^{-1/(p-q)}t_{\lambda_n}^+(v_n)v_n \rightharpoonup u$ in $W_0^{1,p}(\Omega)$ and $\lambda_n^{-1/(p-q)}t_{\lambda_n}^+(v_n)v_n \to u$ in $L^p(\Omega)$, $L^\gamma(\Omega)$ as $n\to \infty$. We claim that $u\neq 0$. Indeed, if not then, $\|v_n\|_q^q\to 0$ and from the Proposition \ref{inequal} item $\mathbf{(ii)}$ we conclude that $\lambda_n^{-1/(p-q)}t_{\lambda_n}^+(v_n)v_n \to 0$ in $W_0^{1,p}(\Omega)$ as $n\to \infty$, however this is an absurd because it implies that $\lim_{n\to \infty}\frac{\hat{J}_{\lambda_n}^+}{\lambda_n^{p/(p-q)}}=0$, which is a contradiction with \eqref{infconv}, therefore $u\neq 0$.
	
	From the equation
	
	$$
	-\Delta_p(\lambda_n^{-1/(p-q)}t_{\lambda_n}^+(v_n)v_n)=(\lambda_n^{-1/(p-q)}t_{\lambda_n}^+(v_n)v_n)^{q-1}+\lambda^{\frac{\gamma-p}{p-q}}(\lambda_n^{-1/(p-q)}t_{\lambda_n}^+(v_n)v_n)^{\gamma-1},
	$$
	and the $S^+$ property of the $p$-Laplacian operator we conclude that $\lambda_n^{-1/(p-q)}t_{\lambda_n}^+(v_n)v_n \to u$ in $W_0^{1,p}(\Omega)$ as $n\to \infty$. Once $u\neq 0$, it follows that $\lambda_n^{-1/(p-q)}t_{\lambda_n}^+(v_n)\to t>0$ and $v_\lambda\to v$ in $W_0^{1,p}(\Omega)$ as $n\to \infty$. From \eqref{infconv} we conclude that 
	\begin{align*}
	\hat{\Phi}_0&=\lim_{n\to \infty}\frac{\hat{J}_{\lambda_n}^+}{\lambda_n^{p/(p-q)}} \\
	&= \lim_{n\to \infty}\left[\frac{1}{p}\left(\frac{t_{\lambda_n}^+(v_{n})}{\lambda_n^{1/(p-q)}}\right)^p-\frac{1}{q}\left(\frac{t_{\lambda_n}^+(v_{n})}{\lambda_n^{1/(p-q)}}\right)^q\|v_{n}\|_q^q-\frac{\lambda_n^{\frac{\gamma-p}{p-q}}}{\gamma}\left(\frac{t_{\lambda_n}^+(v_{n})}{\lambda_n^{1/(p-q)}}\right)^\gamma F(v_{n})\right] \\
	&= \frac{1}{p}t^p-\frac{1}{q}t^q\|v\|_q^q, 
	\end{align*}
	and consequently $v=\hat{z}$ and $t=\|\hat{z}\|_q^{q/(p-q)}$, however this contradicts \eqref{absrd} and thus the Lemma is proved.

\end{proof}

%%%%%%%%%%%%%%%%%%%%%%%%%%%%%%%%%%%%%%%%%%%%%%%%%%%%%%%%%%%%%%%%%%%%%%%%%%%%%%%%%%%%%%%%%%%%%%

\section{Proof of Theorem \ref{THM}}

\begin{proof}  $\mathbf{i)}$ From the Lemmas \ref{loca1} and \ref{loca2}, for each $\lambda\in (0,\varepsilon)$ we can find $0<w_\lambda\in \mathcal{N}_\lambda^-$ and $0<u_\lambda\in \mathcal{N}_\lambda^+$ solutions of \eqref{pq}. Observe from the definitions of $\mathcal{N}_\lambda^-,\mathcal{N}_\lambda^+$ that $D_{uu}\Phi_{\lambda}(w_\lambda)(w_\lambda,w_\lambda)<0$ and $D_{uu}\Phi_{\lambda}(u_\lambda)(u_\lambda,u_\lambda)>0$.
	
	$\mathbf{(ii)}$ From the Lemma \ref{near0} we have that 
	
	$$
	\lim_{\lambda\downarrow 0}\frac{u}{\lambda^{-1/(p-q)}}=z,\ \forall\ u\in \mathcal{S}_\lambda^+.
	$$
	
	Once $u_\lambda\in \mathcal{S}_\lambda^+$, the proof is completed.	
\end{proof}

%%%%%%%%%%%%%%%%%%%%%%%%%%%%%%%%%%%%%%%%%%%%%%%%%%%%%%%%%%%%%%%%%%%%%%%%%%%%%%%%%%%%%%%%%%%%%%%

\appendix

\section{}

\begin{prop}\label{uniconv} There holds
	\begin{description}
		\item[(i)] take $C>0$ and $d>0$. Suppose that $\varepsilon$ is given as in the Corollary \ref{unif3}. There exists a constant $\overline{C}>0$ such that 
		for all $\lambda,\lambda'\in [\lambda^*,\lambda^*+\varepsilon]$ we have
		$$
		|t_{\lambda}^-(w)-t_{\lambda'}^-(w)|\le \overline{C}|\lambda-\lambda'|,\ \forall\ w\in \mathcal{N}_{\lambda^*,d,C}^-;
		$$
		
		\item[(ii)] take $c>0$ and $d>0$. Suppose that $\varepsilon$ is given as in the Corollary \ref{unif3}. There exists a constant $\overline{c}>0$ such that 
		for all $\lambda,\lambda'\in [\lambda^*,\lambda^*+\varepsilon]$ we have
		$$
		|t_{\lambda}^+(u)-t_{\lambda'}^+(u)|\le\overline{c}|\lambda-\lambda'|,\ \forall\ u\in \mathcal{N}_{\lambda^*,d,c}^+.
		$$
	\end{description}
\end{prop}

\begin{proof} $\mathbf{(i)}$ Recall from the Proposition \ref{decrea1} that for all $w\in \mathcal{N}_{\lambda^*,d,C}^-$ we have that
	
	$$
	\frac{\partial }{\partial \lambda}t_\lambda^{-}(w)=\frac{\|t_\lambda^{-}(w)w\|_q^q}{H_\lambda^-(t_\lambda^{-}(w)w)}, \ \forall \lambda\in [\lambda^*,\lambda^*+\varepsilon).
	$$	
	Also from the Proposition \ref{decrea1} we have that $t^-_\lambda(w)\le 1$ for all $w\in \mathcal{N}_{\lambda^*,d,C}^-$ and hence $\|t_{\lambda}^-(w)w\|_q^q\le C$ for all   $w\in \mathcal{N}_{\lambda^*,d,C}^-$. Moreover, from the Corollary \ref{unif3} we have that $H_\lambda^-(w)\le \delta<0$ for each $w\in \mathcal{N}_{\lambda^*,d,C}^-$. Therefore, from the mean value theorem, we conclude that 
	
	$$
	|t_{\lambda}^-(w)-t_{\lambda'}^-(w)|\leq \left|\frac{\partial }{\partial \lambda}t_{\theta}^{-}(w)\right||\lambda-\lambda'|\le \frac{C}{|\delta|}|\lambda-\lambda|',
	$$
	where $\theta\in (\lambda,\lambda')$.
	
	$\mathbf{(ii)}$ Recall from the Proposition \ref{decrea1} that for all $u\in \mathcal{N}_{\lambda^*,d,c}^+$ we have that
	
	$$
	\frac{\partial }{\partial \lambda}t_\lambda^{+}(u)=\frac{\|t_\lambda^{+}(u)u\|_q^q}{H_\lambda^+(t_\lambda^{+}(u)u)}, \ \forall \lambda\in [\lambda^*,\lambda^*+\varepsilon).
	$$	
	Observe from the Corollary \ref{estimatives} that there exists a positive constant $C_1$ such that $t^+_\lambda(u)\le C_1$ for all $u\in \mathcal{N}_{\lambda^*,d,c}^+$ and hence $\|t_{\lambda}^+(u)u\|_q^q\le C_1^qC$ for all   $w\in \mathcal{N}_{\lambda^*,d,c}^+$. Moreover, from the Corollary \ref{unif3} we have that $H_\lambda^+(u)> \delta>0$ for each $u\in \mathcal{N}_{\lambda^*,d,c}^+$. Therefore, from the mean value theorem, we conclude that 
	
	$$
	|t_{\lambda}^+(u)-t_{\lambda'}^+(u)|\leq \left|\frac{\partial }{\partial \lambda}t_{\theta}^{+}(u)\right||\lambda-\lambda'|\le \frac{C_1^qC}{\delta}|\lambda-\lambda|',
	$$
	where $\theta\in (\lambda,\lambda')$.
	
\end{proof}

\begin{table}[!h]

	\begin{center}
		\begin{tabular}{|c|c|l|}
			\hline
			$\Phi_{\lambda}$ & Page \pageref{philambda} \\
			\hline
			$\mathcal{N}_\lambda, \mathcal{N}_\lambda^-, \mathcal{N}_\lambda^+, \mathcal{N}_\lambda^0$ &  Page \pageref{pagen}\\
			\hline 
			$\phi_{\lambda,u}$ & Page \pageref{philambdau} \\ 
			\hline
			$t_\lambda^0$, $t_\lambda^+$, $t_\lambda^-$ & Page \pageref{fibermaps}\\ 
				\hline
				$t(u)$, $\lambda(u)$ &  Page \pageref{rayleigh}\\
				\hline 
		 $\lambda^*$ & Page \pageref{extremal}\\ 
		\hline
		$\hat{\mathcal{N}}_\lambda$, $\hat{\mathcal{N}}_\lambda^+$ &  Page \pageref{N^}\\
		\hline 
				\end{tabular}
	\begin{tabular}{|c|c|l|}
		\hline
		$t_{\lambda^*}$, $s_{\lambda^*}$ &  Page \pageref{t-lambda*}\\
		\hline 
		$J_{\lambda}^-(u)$, $J_{\lambda}^+(u)$, $\hat{J}_{\lambda}^-$, $\hat{J}_{\lambda}^+$&  Page \pageref{jlambda}\\
		\hline
		$\hat{\Phi}_{\lambda^*}^-$, $\hat{\Phi}_{\lambda^*}^+$ &  Page \pageref{hatphi}\\
		\hline
		$H_\lambda^-$, $H_\lambda^+$ & Page \pageref{hlambda-} \\
		\hline
		$\mathcal{N}_{\lambda^*,d,C}^-$, $\mathcal{N}_{\lambda^*,d,C}^+$ & Page \pageref{vacavaca} \\
 		\hline
 		$\mathcal{S}_\lambda^-$, $\mathcal{S}_\lambda^+$ & Page \pageref{boiboi} \\
 		\hline
 		$\hat{J}^-_{\lambda,d^-,C}$, $\hat{J}^+_{\lambda,d^+,c}$ & Page \pageref{jmenoslambdadeceparalembraotrem} \\
 		\hline
 		
	\end{tabular}
	\end{center}
	\label{table1}
		\caption{Main Notations}
	
\end{table}

%%%%%%%%%%%%%%%%%%%%%%%%%%%%%%%%%%%%%%%%

%% Numbered
%\bibliographystyle{model1-num-names}

%% Numbered without titles
%\bibliographystyle{model1a-num-names}

%% Harvard
%\bibliographystyle{model2-names.bst}\biboptions{authoryear}

%% Vancouver numbered
%\usepackage{numcompress}\bibliographystyle{model3-num-names}

%% Vancouver name/year
%\usepackage{numcompress}\bibliographystyle{model4-names}\biboptions{authoryear}

%% APA style
%\bibliographystyle{model5-names}\biboptions{authoryear}

%\usepackage{numcompress}\bibliographystyle{model6-num-names}
\bibliographystyle{elsarticle-num}
\bibliography{Ref}

\end{document}